\documentclass[11pt, reqno]{amsart}
\usepackage{mathrsfs}
\usepackage[active]{srcltx}
\usepackage{mathrsfs,amsmath}
\usepackage{mathtools}
\usepackage{longtable}
\usepackage{todonotes}
\allowdisplaybreaks

\UseRawInputEncoding

\usepackage{xcolor}
\definecolor{bluecite}{HTML}{0875b7}
\usepackage[unicode=true,
bookmarksopen={true},
pdffitwindow=true,
colorlinks=true,
linkcolor=bluecite,
citecolor=bluecite,
urlcolor=bluecite,
hyperfootnotes=false,
pdfstartview={FitH},
pdfpagemode= UseNone]{hyperref}

\newtheorem{proposition}{Proposition}[section]
\newtheorem{theorem}{Theorem}[section]

\newtheorem{corollary}{Corollary}[section]

\newtheorem{remark}{Remark}[section]
\numberwithin{equation}{section}

\usepackage{bbm}  
\usepackage{dsfont}

\address{\textsc{Zolt\'an Balogh}: 
	Mathematisches Institut, University of Bern, Sidlerstrasse 12, 3012, Bern, Switzerland.}
\email{zoltan.balogh@unibe.ch}
\address{\textsc{Alexandru Krist\'aly}: 
Department of Economics, Babe\c s-Bolyai University, str. Teodor Mihali 58-60, 400591, Cluj-Napoca, Romania \& Corvinus Centre for Operations Research, Corvinus Institute for Advanced Studies, Corvinus University of Budapest,  F\H ov\'am t\'er 8, 1093, Budapest, Hungary. 
}
\email{
	 alexandru.kristaly@ubbcluj.ro}
\address{\textsc{Francesca Tripaldi}:
	Mathematisches Institut, University of Bern, Sidlerstrasse 12, 3012, Bern, Switzerland.}
\email{francesca.tripaldi@unibe.ch}

\subjclass[]{ 
	28A25, 
	26D15,
	46E35, 
	49Q22
}
\keywords{Log-Sobolev inequality, sharpness,  metric spaces, hypercontractivity, ${\sf CD}(0,N)$ condition.}
\thanks{Z. M. Balogh and F. Tripaldi were
	supported by the Swiss National Science Foundation, Grant Nr. {200020\_191978}.  Research of 
	A. Krist\'aly  was done while visiting the Corvinus Institute for Advanced Studies, Corvinus University, Budapest, Hungary;  research also supported by the UEFISCDI/CNCS grant PN-III-P4-ID-PCE2020-1001.
}

\title[Sharp log-Sobolev inequalities in ${\sf CD}(0,N)$ spaces]{Sharp log-Sobolev inequalities in ${\sf CD}(0,N)$ spaces  with applications}

\date{\today}

\author[Z.\ M.\ Balogh, A.\ Krist\'aly and F. Tripaldi]{Zolt\'an M. Balogh, Alexandru Krist\'aly and Francesca Tripaldi}

\begin{document}
		\vspace{-1cm}
		\begin{abstract}
		Given  $p,N>1,$ we prove the sharp $L^p$-log-Sobolev inequality on noncompact  metric measure spaces satisfying the ${\sf CD}(0,N)$ condition,    where the optimal constant involves the \textit{asymptotic volume ratio} of the space. This proof is based on a sharp isoperimetric inequality in ${\sf CD}(0,N)$ spaces, symmetrisation,  and a careful scaling argument.  As an application  we establish a sharp hypercontractivity estimate  for the Hopf--Lax semigroup in ${\sf CD}(0,N)$ spaces. 
		  The proof of this result uses  Hamilton--Jacobi inequality and  Sobolev regularity properties of the Hopf--Lax semigroup, which turn out to be essential in the present setting of nonsmooth and noncompact spaces. 
		  Moreover,  a sharp Gaussian-type $L^2$-log-Sobolev inequality and a hypercontractivity estimate are obtained in  ${\sf RCD}(0,N)$ spaces. Our results are new, even in the smooth setting of Riemannian/Finsler manifolds. In particular, an extension of  the celebrated rigidity result of Ni (\textit{J. Geom. Anal.}, 2004) on Riemannian manifolds will be a  simple consequence of our sharp log-Sobolev inequality. 
	\end{abstract}
	\maketitle 
		\vspace{-1cm}
	\tableofcontents

	\section{Introduction and main results}
	
Different forms of the log-Sobolev inequality appear as indispensable tools to  describe nonlinear phenomena, such as the solution of Poincar\'e's conjecture (see Perelman \cite{Perelman}), quantum field theory (see e.g. Glimm and Jaffe \cite{GJ}), hypercontractivity estimates for Hopf--Lax semigroups (see e.g. Bobkov,  Gentil and Ledoux \cite{BobkovGL}, Gentil \cite{Gentil}, Otto and Villani \cite{OV1, OV2}), equilibrium for spin systems (see Guionnet and Zegarlinski \cite{GZ}), or hydrodynamic scalings for systems of interacting particles (see Yau \cite{Yau}). 
	
	 The  optimal Euclidean $L^p$-log-Sobolev inequality for $1 < p < n$  in $\mathbb R^n$
is due to  Del Pino and Dolbeault \cite{delPinoDolbeault-2}, and states that, for every Sobolev function  $u \in {W}^{1,p}(\mathbb R^n)$ with $\displaystyle\int_{\mathbb R^n} |u|^p dx = 1$, one has the inequality 
\begin{equation}\label{e-sharp-log-Sobolev}
	\int_{\mathbb R^n}  |u|^p\log |u|^p dx\leq \frac{n}{p}	\log\left(\mathcal
	L_{p,n}\displaystyle\int_{\mathbb R^n} |\nabla u|^p dx\right).
\end{equation}
In the above inequality, we used the notation $\mathcal{L}_{p,N}$ for any two given numbers $p,N>1$ to denote 	
$$		\mathcal
		L_{p,N}={\small
			\frac{p}{N}\left(\frac{p-1}{e}\right)^{p-1}\left(\sigma_N{\Gamma\left(\frac{N}{p'}+1\right)}\right)^{-\frac{p}{N}}}, 
$$
where $\sigma_N=\frac{\pi^\frac{N}{2}}{\Gamma(\frac{N}{2}+1)},$  $\Gamma$ being the usual Euler Gamma-function, and $p'=\frac{p}{p-1}$ is the conjugate of $p$. Note that $\mathcal L_{p,N}$ is well-defined for any $N>1$, and not just for integer values. We also mention that 
 the constant $\mathcal L_{p,n}$ in the inequality \eqref{e-sharp-log-Sobolev} is sharp. 
 This  inequality  appeared first in the paper of Weissler \cite{Weissler} for $p=2$, which is equivalent to the dimension-free log-Sobolev inequality of Gross \cite{Gross}, and it was stated for the Gaussian probability measure $\gamma_G$ with $d\gamma_G(x)=\frac{1}{(2\pi)^\frac{n}{2}}e^{-\frac{|x|^2}{2}}dx$, i.e., for every function  $u \in {W}^{1,2}(\mathbb R^n,\gamma_G)$ with $\displaystyle\int_{\mathbb R^n} u^2 d\gamma_G = 1$, one has 
\begin{equation}\label{Gross-inequality}
	\int_{\mathbb R^n}  u^2\log u^2 d\gamma_G\leq 2	\displaystyle\int_{\mathbb R^n} |\nabla u|^2 d\gamma_G.
\end{equation} 
Soon after the paper \cite{delPinoDolbeault-2} was published, 
Gentil \cite{Gentil}  by the Pr\'ek\'opa--Leindler inequality, and Agueh, Ghoussoub and  Kang \cite{AGK} by optimal mass transportation, showed  that \eqref{e-sharp-log-Sobolev} is valid for all $p > 1$.

 Moreover, the optimal transport method was applicable even in the case of curved structures, where the influence of the Ricci curvature plays an important role. Indeed, by optimal transport arguments, Cordero-Erausquin, McCann and Schmuckenschl\"{a}ger established in \cite{CEMS2} (see also Cordero-Erausquin \cite{Cordero})  a Bakry--\'Emery-type estimate   \cite{BE}, i.e., if $(M, g)$ is a complete Riemannian manifold,  $V: M \to \mathbb R$ is a smooth enough function, and $K >0$ that satisfy the curvature condition 
$$ {\rm Hess}_x V + {\rm Ric}_x \geq K\, {\rm Id}\ \  \ {\rm for\ all} \ x \in M,$$
then, for the weighted Riemannian measure $d \gamma_V = e^{-V} d{\rm vol}$, the inequality 
\begin{equation}\label{curved-Gross-inequality}
	\int_{M}  u^2\log u^2 d\gamma_V\leq \frac{2}{K}	\displaystyle\int_{M} |\nabla u|^2 d\gamma_V,
\end{equation} 
holds for any  $u\in W^{1,2}(M,\gamma_V)$ such that $\displaystyle\int_M u^2 d\gamma_V = 1$.\ 
In the Euclidean case where $M= \mathbb{R}^n$, one can obtain the Gaussian log-Sobolev inequality \eqref{Gross-inequality} from \eqref{curved-Gross-inequality}, by choosing $K =1$  and $V(x)= \frac{1}{2}|x-x_0|^2$ for some $x_0\in M$. However, it is not possible to recover by \eqref{curved-Gross-inequality} the optimal Euclidean $L^2$-log-Sobolev inequality \eqref{e-sharp-log-Sobolev} by choosing $V\equiv c$. In this respect, note that also in the work of Barthe and Kolesnikov \cite{BartheK} general versions of log-Sobolev inequalities were proven for Riemannian manifolds for measures with the tail behavior of order $e^{-|x|^a}$ for some $a>0$ whenever $|x|\gg 1$ . 
Moreover, we observe that the multiplicative constant $\frac{2}{K}$ on the right side of \eqref{curved-Gross-inequality} blows up as $K \to 0$, indicating that it might be difficult (or even impossible) to find an appropriate version of \eqref{curved-Gross-inequality} for general manifolds with $0$ lower bound on their Ricci curvature.  

This is precisely the main purpose of our paper; namely, to prove a \textit{sharp} version of the  $L^p$-log-Sobolev inequality \eqref{e-sharp-log-Sobolev} (and eventually a corresponding version of the  Gaussian log-Sobolev inequality \eqref{Gross-inequality}) for a general metric measure space $(X,d, {\sf m})$ satisfying the curvature-dimension condition ${\sf CD}(0,N)$ in the sense of Lott--Sturm--Villani (see \cite{LV, Sturm-1, Sturm-2}). This class of metric measure spaces contains as particular examples Riemannian/Finsler manifolds with non-negative Ricci curvature and their Gromov--Hausdorff limits. Metric cones studied by Bacher and Sturm \cite{BS} and Ketterer \cite{Ketterer} are also included in this class; for more examples, we refer to \cite{BK}. 

In order to formulate our result, we need to introduce an additional notion. Given a metric measure space $(X,d, {\sf m})$ that satisfies  the curvature-dimension condition ${\sf CD}(0,N)$, according to Bishop--Gromov theorem, see Sturm \cite{Sturm-2}, one can define the \textit{asymptotic volume ratio}
$${\sf AVR}_ {\sf m}=\lim_{r\to \infty}\frac{ {\sf m}(B(x,r))}{\sigma_Nr^N},$$
which is independent of the choice of $x\in X$. 
Let $W^{ 1,p}(X,d,{\sf m})$ be the space of real-valued Sobolev functions over $X$, and $|\nabla u | \in  L^p( {X,\sf m})$ be the minimal $p$-weak upper gradient of $u \in W^{ 1,p}(X,d,{\sf m}),$  which exists  $ {\sf m}$-a.e.\ on $X$, see \S \ref{section2}; these notions were introduced by Heinonen and Koskela \cite{HK} (see also Hajlasz \cite{Hajlasz} and Shanmugalingam  \cite{Shanmugalingam}). These concepts were extensively studied by Cheeger \cite{Cheeger} in the context of  differentiable structures of metric measures spaces. They play an important role in the study of Hamilton--Jacobi equations in the general metric measure setting, see Ambrosio,  Gigli and Savar\'e  \cite{AGS, AGS-2}.

	  Our first main result provides the natural and sharp extension of \eqref{e-sharp-log-Sobolev} to the class of ${\sf CD}(0,N)$ spaces. 
	
	\begin{theorem}\label{log-Sobolev-main} Let $N,p>1$, and 
		 $(X,d, {\sf m})$ be a ${\sf CD}(0,N)$ space with 
		${\sf AVR}_ {\sf m}>0.$
		 Then, for every  
		$ u\in W^{ 1,p}(X,d,{\sf m})$ with $\displaystyle\int_X |u|^pd  {\sf m}=1$, one has that 
		\begin{equation}\label{LSI}
			\int_{X}|u|^p\log |u|^pd {\sf m}\leq \frac{N}{p}\log\left(\mathcal
			L_{p,N}{\sf AVR}_ {\sf m}^{-\frac{p}{N}}\int_X |\nabla u|^pd {\sf m} \right). 
		\end{equation}
	In addition, the constant  $\mathcal
	L_{p,N}{\sf AVR}_ {\sf m}^{-\frac{p}{N}}$ in \eqref{LSI} is sharp$:$ if there exists $C>0$ such that for every  
	$ u\in W^{ 1,p}(X,d,{\sf m})$  with $\displaystyle\int_X |u|^pd  {\sf m}=1$ one has
	\begin{equation} \label{LSI-CCC}
		\int_{X}|u|^p\log |u|^pd {\sf m}\leq \frac{N}{p}\log\left(C\int_X |\nabla u|^pd {\sf m} \right),
	\end{equation}
	then
		\begin{equation} \label{C-estimate}
		C \geq  \mathcal L_{p, N}{\sf AVR}_ {\sf m}^{-\frac{p}{N}}.
		\end{equation}
	
	\end{theorem}
	
Let us note that in the case of the Euclidean space we have ${\sf AVR}_ {\sf m}= {\sf AVR}_ {\sf \mathcal{L}^n}= 1$ and thus \eqref{LSI} is a direct generalization of \eqref{e-sharp-log-Sobolev}. In the case of a Riemannian manifold $(M,g)$ (endowed with its canonical measure $dv_g$), one has  $0\leq {\sf AVR}_ {dv_g} \leq 1$, which makes the constant on the right side of \eqref{LSI} slightly worse than in the Euclidean case. In particular, Theorem \ref{log-Sobolev-main} implies a generalisation of the  rigidity result of Ni \cite{Ni} in the Riemannian context;  see  Remark \ref{remark-Ni} for details. 
	
We mention at this point that certain versions of log-Sobolev inequalities proven by Lott and Villani in \cite{LV} for ${\sf CD}(K, N)$ spaces with $K>0$ (see also Bakry,  Gentil and Ledoux \cite{BakryGL}) are similar to \eqref{curved-Gross-inequality},  as they also feature a multiplicative factor $\frac{1}{K}$ which goes to $\infty$ as $K \to 0$. As previously observed, this supports the fact that it is not possible to obtain a sharp log-Sobolev inequality by a simple limiting argument, without the additional assumption that ${\sf AVR}_ {\sf m}>0$. Furthermore, one can see the second statement of Theorem \ref{log-Sobolev-main} as a converse of the first one in  the sense that if $(X,d, {\sf m})$ is a ${\sf CD}(0,N)$  space that supports an $L^p$-log-Sobolev inequality of the type \eqref{LSI-CCC}, then ${\sf AVR}_ {\sf m}>0$, see \eqref{C-estimate}. Therefore, the condition ${\sf AVR}_ {\sf m}>0$ is a necessary and sufficient condition for the validity of an Euclidean-type log-Sobolev inequality in the framework of ${\sf CD}(0, N)$ spaces with a precise relationship between the optimal log-Sobolev constant and the value ${\sf AVR}_ {\sf m}>0$; for a more precise statement, see Remark \ref{remark-equiv}. We notice that Bakry and Ledoux \cite[Theorem 3]{B-L} proved an $L^2$-log-Sobolev inequality for diffusion operators satisfying a certain ${\sf CD}(0, N)$ condition in terms of the Bakry-\'Emery \textit{carr\'e de champ} on Riemannian manifolds under a special condition on the generator of the diffusion operator. However, this results does not capture the  ${\sf AVR}_ {\sf m}$ constant as the Euclidean constant appears on the right side of the log-Sobolev inequality. 

The proof of \eqref{LSI} combines several known results by now: (a) the recent sharp isoperimetric inequality in ${\sf CD}(0,N)$ spaces, proved by Balogh and Krist\'aly \cite{BK}; (b)  a sharp weighted log-Sobolev inequality on cones, see Balogh, Don and Krist\'aly \cite{BDK}, and (c) P\'olya--Szeg\H o type  rearrangement arguments on ${\sf CD}(0,N)$ spaces, due to Nobili and Violo \cite{NV, NV2}, see also Mondino and Semola \cite{MSemola}. 
	In order to prove the inequality \eqref{LSI}, a  nonnegative function $ u\in W^{ 1,p}(X,d,{\sf m})$ is rearranged with respect to the $1$-dimensional model space $[0,\infty)$ endowed with the measure $\omega=N\sigma_Nr^{N-1}\mathcal L^1$, see \S \ref{subsection212} for details. Since the entropy-term remains invariant under such rearrangement, the P\'olya--Szeg\H o inequality involving the term ${\sf AVR}_ {\sf m}$, see \cite{NV, NV2} (which follows from a suitable co-area formula and the isoperimetric inequality from \cite{BK}), and the sharp weighted log-Sobolev inequality on the model cone $([0,\infty), |\cdot|,\omega) $ imply inequality in \eqref{LSI}. We notice that the limit case $p=\infty$ in \eqref{LSI} can also be obtained; a similar result is established by Fujita \cite{Fujita} in the Euclidean case. 
	
	The proof of the \textit{sharpness} of the constant $\mathcal
	L_{p,N}{\sf AVR}_ {\sf m}^{-\frac{p}{N}}$ in \eqref{LSI} is more complicated, where a careful 
	transition from the original measure ${\sf m}$ to the measure $\omega=N\sigma_Nr^{N-1}\mathcal L^1$ is performed, combined with a subtle 
	scaling argument which produces the factor  ${\sf AVR}_ {\sf m}$. We notice that an alternative (but less general) proof of the same sharpness is known from \cite{Kristaly-Calculus}, where some robust  ordinary differential equations/inequalities are compared.

	The first nontrivial application of Theorem \ref{log-Sobolev-main} is a sharp  hypercontractivity estimate for the Hopf--Lax semigroup on metric measure spaces $(X,d, {\sf m})$ satisfying the ${\sf CD}(0,N)$ condition. 
		Let $p>1$; for  $t>0$ and  $u:X\to \mathbb R$  we consider the Hopf--Lax formula 
	$$		{\bf Q}_{t}u(x)\coloneqq\inf_{y\in X}\left\{u(y)+\frac{d(x,y)^{p'}}{p't^{p'-1}}\right\}, \ x\in X.
	$$
	 By convention, ${\bf Q}_{0}u=u.$ 
	 
	 Before stating the result, let us notice that  when  $(X,d, {\sf m})$  is a ${\sf CD}(K,N)$ space with $K>0$ (thus, $X$ is compact) and $u$ is a bounded Lipschitz function on $X$, one has that 	${\bf Q}_{t}u(x)\in \mathbb R$ for every $(t,x)\in (0,\infty)\times X$. Note however that the limit case ${\sf CD}(0,N)$ requires finer properties on the function $u:X\to \mathbb R$; 
	 indeed, if we still keep the above class of functions and ${\sf m}(X)=+\infty$ (i.e., ${\sf m}$ cannot be normalized to a probability measure), then the $L^\alpha(X,{\sf m})$-norm of $e^u$ (which appears in the hypercontractivity estimate, see \eqref{hyperc-estimate} below) would be $+\infty$ for every $\alpha>0$.  Thus,  in order to provide a meaningful  hypercontractivity estimate for  Hopf--Lax semigroups on ${\sf CD}(0,N)$ spaces,  a suitable class of functions is needed.
	 
	 Accordingly, for some $t_0>0$ and $x_0\in X$, we shall consider the class of functions $u:X\to \mathbb R$, denoted by $\mathcal  F_{t_0,x_0}(X)$, which satisfy the following assumptions: 
	\begin{itemize}
		\item[(A1)] $u\in {\rm Lip}_{\rm loc}(X)$ and the set $u^{-1}([0,\infty))\subset X$ is bounded; 
			\item[(A2)]   ${\bf Q}_{t_0}u(x_0)>-\infty;$
	 	\item[(A3)] there exist $M\in \mathbb R$ and $C_0>p't_0^{p'-1}$ such that $$u(x)\geq M-\frac{d^{p'}(x,x_0)}{C_0}\ ,\ \ \forall x\in X.$$
	\end{itemize}
One can prove that ${\bf Q}_{t}u(x) \in \mathbb R$ for every $(t,x)\in (0,t_0)\times X$, whenever $u \in \mathcal  F_{t_0,x_0}(X)$, see Proposition \ref{proposition-hyper}. 
Our second result provides the following sharp  hypercontractivity estimate on  ${\sf CD}(0,N)$ spaces. 
	
	\begin{theorem}\label{Hopf--Lax-theorem}
		Let  $N>1$ and 
		$(X,d, {\sf m})$ be a ${\sf CD}(0,N)$ space with 
		${\sf AVR}_ {\sf m}>0$. Let $t_0>0$ and $x_0\in X$ be fixed. 	
		 Then for every $p>1$,  $0<\alpha\leq \beta$,  $ t\in (0,t_0)$ and  $u\in \mathcal F_{t_0,x_0}(X)$ with $(1+d^{p'}(x_0,\cdot))e^{\alpha u}\in  L^1(X,{\sf m})$ and $e^\frac{\alpha u}{p}\in   W^{ 1,p}(X,d,{\sf m})$,  we have that
		\begin{equation}\label{hyperc-estimate}
			\|e^{{\bf Q}_{t}u}\|_{L^{\beta}(X,{\sf m})}\leq \|e^{u}\|_{L^{\alpha}(X,{\sf m})}\left(\frac{\beta-\alpha}{ t}\right)^{\frac{N}{p}\frac{\beta-\alpha}{\alpha\beta}}\frac{\alpha^{\frac{N}{\alpha\beta}(\frac{\alpha}{p}+\frac{\beta}{p'})}}{\beta^{\frac{N}{\alpha\beta}(\frac{\beta}{p}+\frac{\alpha}{p'})}}\left(\mathcal
			L_{p,N}{\sf AVR}_ {\sf m}^{-\frac{p}{N}}\frac{Ne^{p-1}}{p^p}\right)^{\frac{N}{p}\frac{\beta-\alpha}{\alpha\beta}}.
		\end{equation} 
		In addition, the constant  $\mathcal
		L_{p,N}{\sf AVR}_ {\sf m}^{-\frac{p}{N}}\frac{Ne^{p-1}}{p^p}$ in \eqref{hyperc-estimate} is sharp whenever $\alpha<\beta$. 
	\end{theorem}
	

	We  notice that the equivalence of the log-Sobolev inequality and hypercontractivity of the Hopf--Lax semigroup has already been studied in the Euclidean setting by Gentil \cite{Gentil}, and Bobkov, Gentil and Ledoux \cite{BobkovGL}. The link between these two results is based on the fact that the Hopf--Lax semigroup satisfies the Hamilton--Jacobi equation.  For compact metric measure spaces, the same link was established by Lott and Villani \cite{ LV-JMPA} and Gozlan, Roberto and Samson \cite{GRS}. In this case, the Hopf--Lax semigroup satisfies a Hamilton--Jacobi inequality that is enough for the purpose of the equivalence. 
	
	 In our setting the same general strategy will be applied. 
 However, in the present situation of nonsmooth and noncompact spaces the implementation of this general strategy becomes  rather subtle. We need to make sure that we are working in the right function class. For example if we choose $u$ in the class of  bounded functions, then the right side of  \eqref{hyperc-estimate} will be infinity and so the claim will be trivially true. This  kind of problem motivated us to search for a suitable class of functions with precise \textit{regularity and growth properties} in order to obtain meaningful estimates. The choice of the function class as described above allows us to establish good properties of  integrals involving the Hopf--Lax semigroup 	${\bf Q}_{t}u$, see Propositions \ref{proposition-hyper} \& \ref{proposition-hyper-2}, that will enable us to carry out the strategy of the proof. Another result that we needed in the proof was the fact that  a Hamilton--Jacobi inequality is satisfied by the Hopf--Lax semigroup in metric measure spaces that was shown  by Ambrosio, Gigli and Savar\'e \cite{AGS}. 
		
	The second application of Theorem \ref{log-Sobolev-main} is an $L^2$-Gaussian log-Sobolev inequality; it turns out that our approach requires the framework of metric measure spaces satisfying the Riemannian curvature-dimension condition ${\sf RCD}(0,N)$, studied in great detail by 	Ambrosio, Gigli and Savar\'e \cite{AGS-Duke}, Ambrosio, Mondino and Savar\'e \cite{AMondinoSavare},
	Erbar, Kuwada and Sturm \cite{EKS}  and Cavalletti and Milman \cite{Cavalletti-MIlman}; for details,    see \S \ref{Gaussian-section}, and Ambrosio \cite{Ambrosio} and  Gigli \cite{ Gigli2} for further insights on the theory of ${\sf RCD}(0,N)$ spaces. 
	In order to state the result, for a fixed point $x_0\in X$, we consider  the ${\sf m}$-\textit{Gaussian probability measure} 
	$$
	d{\sf m}_{G,x_0}(x):=G^{-1}e^{-\frac{d^2(x_0,x)}{2}}d{\sf m}(x)\ \ {\rm with}\ \ \displaystyle G=\int_Xe^{-\frac{d^2(x_0,x)}{2}}d {\sf m}(x),
	$$ and the ${\sf m}$-\textit{density at} $x\in X,$ i.e., 
	\begin{equation}\label{m-density}
			\theta_{\sf m}(x)=\lim_{r\to 0}\frac{ {\sf m}(B(x,r))}{\sigma_Nr^N}.
	\end{equation}
 A simple consequence of the Bishop--Gromov comparison principle  is that 
	$$\theta_ {\sf m}(x)\geq {\sf AVR}_ {\sf m},\ \ \forall x\in X.$$ 
	The  $L^2$-Gaussian log-Sobolev inequality in ${\sf RCD}(0,N)$ spaces reads as follows: 
		
	\begin{theorem}\label{Gaussian-theorem}
		Let $(X,d, {\sf m})$ be an ${\sf RCD}(0,N)$ space with $N>1$ and assume that ${\sf AVR}_ {\sf m}>0$. Given $x_0\in X$,  for any $u\in W^{1,2}(X,d,{\sf m}_{G,x_0})$ such that  $\displaystyle \int_Xu^2 d {\sf m}_{G,x_0}=1$, one has
		\begin{equation}\label{Gaussian-log-Sobolev}
			\int_Xu^2\log u^2 d {\sf m}_{G,x_0}\le 2\int_X\vert\nabla u\vert^2 d {\sf m}_{G,x_0}+\log\frac{\theta_{\sf m}(x_0)}{{\sf AVR}_ {\sf m}}.
		\end{equation}
	In addition, the constant $2$ is sharp in \eqref{Gaussian-log-Sobolev}. 
	\end{theorem}
In the  Euclidean case with the Lebesgue measure, i.e.,  $(\mathbb{R}^n,\vert\cdot\vert,\mathcal{L}^n)$,   we have that $\theta_{\mathcal{L}^n}(x_0)=1$  for any $x_0\in \mathbb{R}^n$, $G=(2\pi)^\frac{n}{2}$ and $\sf{AVR}_{\mathcal{L}^n}=1$, so we recover \eqref{Gross-inequality} as a particular case of \eqref{Gaussian-log-Sobolev}. 

Note also how the statement of the above result might depend qualitatively on the choice of the basepoint $x_0 \in X$. Indeed, e.g.\ for the model space  $(X, d, {\sf m})= ([0, \infty), |\cdot |, \omega)$ with $\omega=N\sigma_Nr^{N-1}\mathcal L^1$, we have that $\theta_{\sf m}(x_0)= \infty $ for $x_0 \neq 0$ and $\theta_{\sf m}(x_0) =1$ when $x_0=0$; therefore, the statement of Theorem \ref{Gaussian-theorem} becomes meaningful only by choosing $x_0=0$. 

 We pointed out that our argument is based on the  infinitesimal Hilbertianity of the space which is guaranteed by   ${\sf RCD}(0,N)$ structures. It is known that the constant $2$ in  \eqref{Gross-inequality} is sharp, and the equality case has been characterized by Carlen \cite[Theorem 5]{Carlen}; here,  we  provide a new proof -- based on scalings --  of the sharpness of  2 in the inequality  \eqref{Gaussian-log-Sobolev} on  generic ${\sf RCD}(0,N)$ spaces.

We also notice the comparability of \eqref{Gaussian-log-Sobolev} and \eqref{Gross-inequality}  (resp. \eqref{curved-Gross-inequality}) as well as the dimension-free character of \eqref{Gaussian-log-Sobolev}, which contains only a quantitative information about the space $(X,d, {\sf m})$ in terms of the `amplitude' of the volume ratio $r\mapsto \frac{ {\sf m}(B(x,r))}{\sigma_Nr^N}$, appearing in the form $\frac{\theta_{\sf m}(x_0)}{{\sf AVR}_ {\sf m}}\geq 1$. Comparing the inequality \eqref{curved-Gross-inequality} with the  expression of \eqref{Gaussian-log-Sobolev}, we can see that the presence of the additive constant on the right side of \eqref{Gaussian-log-Sobolev} compensates for the lack of positive lower bound on the Ricci curvature; in this sense we can interpret the   ${\sf AVR}_ {\sf m}>0$ condition as a kind of ``positive curvature at infinity''.  Moreover, it turns out that if $\theta_ {\sf m}(x_0)= {\sf AVR}_ {\sf m}$ for some $x_0\in X$, then we are in the rigid setting where `volume cone implies metric cone',  i.e., the space  $(X,d, {\sf m})$ is isometric to a certain model metric cone; for details, see De Philippis  and   Gigli \cite{DP-G}.  Due to the defective   log-Sobolev inequality \eqref{Gaussian-log-Sobolev}, a hypercontractivity bound for the ${\sf m}$-Gaussian measure is also stated.

The paper is organized as follows. In Section \ref{section2}, we provide those notions and results that are indispensable to prove our main results; namely, basic properties of ${\sf CD}(0,N)$ spaces (Bishop--Gromov comparison principle and isoperimetric inequality), rearrangement arguments and P\'olya--Szeg\H o inequality on ${\sf CD}(0,N)$  spaces, as well as a sharp log-Sobolev inequality on the $1$-dimensional weighted model space. Section \ref{section3} is devoted to the proof of Theorem \ref{log-Sobolev-main}, most of the arguments being focused on the proof of the sharpness of the log-Sobolev inequality \eqref{LSI}. In Section \ref{section4} we prove the sharp hypercontractivity estimate in ${\sf CD}(0,N)$ spaces (i.e., Theorem \ref{Hopf--Lax-theorem}). In Section  \ref{Gaussian-section} we recall some basic properties of ${\sf RCD}(0,N)$ spaces (e.g. the Laplace comparison for the distance function) which are needed to prove the sharp $L^2$-Gaussian log-Sobolev inequality in ${\sf RCD}(0,N)$ spaces (see Theorem \ref{Gaussian-theorem}). Finally, Section \ref{section-final} is devoted to presenting some problems which are closely related to our results, and represent further open questions.

		\section{Preliminaries}\label{section2}
	\subsection{Properties of ${\sf CD}(0,N)$ spaces}

\subsubsection{${\sf CD}(0,N)$ spaces, Bishop--Gromov comparison principle and isoperimetric inequality}

Let us fix $(X,d, {\sf m})$  a metric measure space, i.e.,
$(X,d)$ is a complete separable metric space and 
$ {\sf m}$ is a locally finite measure on $X$ endowed with its
Borel $\sigma$-algebra. 

For $p \geq 1$ we denote by $L^p(X,{\sf m})=\left\{ u\colon X\to\mathbb{R}: u\text{ is measurable, }\displaystyle\int_X\vert u\vert^p\,d{\sf m}<\infty\right\}, $ the set of $p$-integrable functions; as usual, functions in the latter definition which are equal {\sf m}-a.e.\ are  identified. The notation  ${\rm Lip}_c(X)$ stands for the space of real-valued Lipschitz functions with compact support in $X$ and ${\rm Lip}_{\rm loc}(X)$ is the space of real-valued locally Lipschitz functions on $X$.

Let	$P_2(X,d)$ be the
	$L^2$-Wasserstein space of probability measures on $X$, while
	$P_2(X,d, {\sf m})$ is  the subspace of
	$ {\sf m}$-absolutely continuous measures on $X$.
For  $N> 1,$ let 
${\rm Ent}_N(\cdot| {\sf m}):P_2(X,d)\to \mathbb R$ be  
the {\it R\'enyi entropy functional} with
respect to the measure $ {\sf m}$, defined by 
\begin{equation}\label{entropy}
	{\rm Ent}_N(\nu| {\sf m})=-\int_X \rho^{-\frac{1}{N}}{\rm d}\nu=-\int_X \rho^{1-\frac{1}{N}}{\rm d} {\sf m},
\end{equation}
where $\rho$ denotes the density function of $\nu^{\rm ac}$ in
$\nu=\nu^{\rm ac}+\nu^{\rm s}=\rho  {\sf m}+\nu^{\rm s}$, where $\nu^{\rm ac}$ and $\nu^{\rm s}$
represent the absolutely continuous and singular parts of $\nu\in
P_2(X,d),$ respectively.

The \textit{curvature-dimension condition} ${\sf CD}(0,N)$
states that for all $N'\geq N$ the functional ${\rm Ent}_{N'}(\cdot|\,u)$ is
convex on the $L^2$-Wasserstein space $P_2(X, d, {\sf m})$, i.e.,  for each
$ {\sf m}_0, {\sf m}_1\in  P_2(X,{d}, {\sf m})$ there exists
a geodesic
$\Gamma:[0,1]\to  P_2(X,{d}, {\sf m})$ joining
$ {\sf m}_0$ and $ {\sf m}_1$ such that for every $s\in [0,1]$ one has 
$${\rm Ent}_{N'}(\Gamma(s)| {\sf m})\leq (1-s) {\rm Ent}_{N'}( {\sf m}_0| {\sf m})+s {\rm Ent}_{N'}( {\sf m}_1| {\sf m}),$$
see Lott and Villani \cite{LV} and Sturm \cite{Sturm-2}.

The \textit{Minkowski content} of $\Omega\subset X$ is given by
\begin{equation}\label{minkowski-content}
{ {\sf m}}^+(\Omega)=\liminf_{\varepsilon \to 0^+}\frac{ {\sf m}(\Omega_\varepsilon\setminus \Omega)}{\varepsilon},
\end{equation}
where $\Omega_\varepsilon=\{x\in X:\exists\, y\in \Omega\ {\rm such\ that}\ d(x,y)< \varepsilon\}$ is the $\varepsilon$-neighborhood of $\Omega$ w.r.t.  the metric $d$.

If $(X,d, {\sf m})$ is a metric measure space satisfying the ${\sf CD}(0,N)$ condition for some $N>1$, then the Bishop--Gromov comparison principle states that both functions 
$$r\mapsto \frac{ {\sf m}^+(B(x,r))}{r^{N-1}}\ \ {\rm and}\ \ r\mapsto \frac{ {\sf m}(B(x,r))}{r^{N}},\ \ r>0,$$
are non-increasing on $[0,\infty)$ for every $x\in X$, see Sturm \cite{Sturm-2}, where $B(x,r)=\{y\in X: d(x,y)<r\}$ is the ball with center $x\in X$ and radius $r>0.$ By the latter Bishop--Gromov monotonicity, the \textit{asymptotic volume ratio}
 $${\sf AVR}_ {\sf m}=\lim_{r\to \infty}\frac{ {\sf m}(B(x,r))}{\sigma_Nr^N},$$
 is well-defined, i.e., it is independent of the choice of $x\in X$.

If ${\sf AVR}_ {\sf m}>0 $, we have the \textit{sharp isoperimetric inequality }on $(X,d, {\sf m})$, namely, for  every bounded Borel subset $\Omega\subset X$ it holds  
\begin{equation}\label{eqn-isoperimetric}
	 {\sf m}^+(\Omega)\geq N\sigma_N^\frac{1}{N}{\sf AVR}_ {\sf m}^\frac{1}{N} {\sf m}(\Omega)^\frac{N-1}{N},
\end{equation}
and the constant $N\sigma_N^\frac{1}{N}{\sf AVR}_ {\sf m}^\frac{1}{N}$ in  \eqref{eqn-isoperimetric} is sharp, see Balogh and Krist\'aly \cite{BK}. We notice that \eqref{eqn-isoperimetric} has been stated already in the smooth setting of Riemannian manifolds by 	Agostiniani,  Fogagnolo and  Mazzieri \cite{Agostiniani-etal, Fogagnolo-Maz-JFA}, Brendle \cite{Brendle} and Johne \cite{Johne}. 
 By using an approximation argument, see Ambrosio, Di Marino and Gigli \cite[Theorem 3.6]{ADMG}, or Nobili and Violo \cite[Proposition 3.10]{NV},  relation \eqref{eqn-isoperimetric} implies that for every open Borel set $\Omega\subset X$ one has
\begin{equation}\label{eqn-isoperimetric-2}
	{\rm Per}(\Omega)\geq N\sigma_N^\frac{1}{N}{\sf AVR}_ {\sf m}^\frac{1}{N} {\sf m}(\Omega)^\frac{N-1}{N},
\end{equation}
where 
$${\rm Per}(\Omega)=\inf\left\{\liminf_{n\to \infty} \int_X |D u_n| d {\sf m}:u_n\in {\rm Lip}_{\rm loc}(X),\ u_n\to\chi_\Omega\ {\rm in}\ L^1_{\rm loc}(X) \right\}$$
denotes the \textit{perimeter} of $\Omega$ in $X$, see Ambrosio and Di Marino \cite{ADM-JFA} and Miranda \cite{Miranda}; here, we  used the 
notion of \textit{local Lipschitz constant} $|Du|(x)$ at $x\in X$  for  $u\in {\rm Lip}_{\rm loc}(X)$,  i.e.,
$$|Du|(x)=\limsup_{y\to x}\frac{|u(y)-u(x)|}{d(x,y)}.$$

	\subsubsection{Rearrangement and P\'olya--Szeg\H o inequality on ${\sf CD}(0,N)$  spaces}\label{subsection212}
	
Let $(X,d, {\sf m})$ be a metric measure space satisfying the ${\sf CD}(0,N)$ condition for some $N>1$ and ${\sf AVR}_ {\sf m}>0.$ We first introduce the natural Sobolev space over $(X,d, {\sf m})$, following e.g.\  Cheeger \cite{Cheeger} and Ambrosio,  Gigli and Savar\'e \cite{AGS}. Let $p>1$. The $p$-\textit{Cheeger energy} ${\sf Ch}_p:L^p(X,{\sf m})\to [0,\infty]$ is defined as the convex and lower semicontinuous functional 
$${\sf Ch}_p(u)=\inf\left\{\liminf_{n\to \infty} \int_X |D u_n|^p d {\sf m}:(u_n)\subset {\rm Lip}(X)\cap L^p(X,{\sf m}),\ u_n\to u\ {\rm in}\ L^p(X,{\sf m})  \right\}.$$
Then  $$W^{ 1,p}(X,d,{\sf m})=\{u\in L^p(X,{\sf m}):{\sf Ch}_p(u)<\infty\}$$ is  the $p$-Sobolev space over $(X,d, {\sf m})$, endowed with the norm 
$\|u\|_{W^{ 1,p}}=\left(\|u\|^p_{L^p(X,{\sf m})}+{\sf Ch}_p(u)\right)^{1/p}. $
Note that $W^{ 1,p}(X,d,{\sf m})$ is a Banach space, which turns out to be also reflexive (due to the fact that $(X,d, {\sf m})$ is doubling), see Ambrosio,  Colombo and  Di Marino \cite{Ambrosio-Colombo}. By the relaxation of the $p$-Cheeger energy,  one can define the minimal $ {\sf m}$-a.e. object $|\nabla u | \in  L^p(X, {\sf m})$, the so-called  \textit{minimal $p$-weak upper gradient of} $u \in W^{ 1,p}(X,d,{\sf m}),$ such that 
$${\sf Ch}_p(u)= \int_X |\nabla u|^p\, d {\sf m}.$$

Since $(X,d, {\sf m})$ is a ${\sf CD}(0,N)$ space, it is doubling (by Bishop--Gromov comparison principle) and  supports the $(1,p)$-Poincar\'e inequality for every $p\geq 1$, see Rajala \cite{Rajala}; thus,  for every $u\in {\rm Lip}_{\rm loc}(X)$ one has  
\begin{equation}\label{Cheeger-equality}
	|Du|(x)=|\nabla u |(x),\ {\sf m}{\rm -a.e.} \ x\in X,
\end{equation}
see Cheeger \cite[Theorem 6.1]{Cheeger}, and 
\begin{equation}\label{Sobolev-second-definition}
W^{ 1,p}(X,d,{\sf m})=\overline{{\rm Lip}_c(X)}^{\|\cdot\|_{W^{ 1,p}}}=\overline{\mathcal W^{ 1,p}(X,d,{\sf m})}^{\|\cdot\|_{W^{ 1,p}}},
\end{equation}
where $\mathcal W^{ 1,p}(X,d,{\sf m})=\{u\in {\rm Lip}_{\rm loc}(X)\cap L^p(X,{\sf m}) :|Du|\in L^p(X,{\sf m})\},$ see e.g.\ Gigli \cite[Theorem 2.8]{Gigli1}. By \eqref{Cheeger-equality}, one has for  any fixed $x_0\in X$ the eikonal equation, i.e., 
\begin{equation}\label{eikonal}
	|\nabla d_{x_0}|=1\ \ {\sf m}{\rm -a.e.},
\end{equation}
where $d_{x_0}(x):=d(x_0,x)$, $x\in X$;  
 see also  Gigli \cite[Theorem 5.3]{Gigli1}.

For a measurable function $u\colon X\to[0,\infty)$, we consider its non-increasing rearrangement $\Hat{u}\colon[0,\infty)\to[0,\infty)$ defined on the 1-dimensional model space $([0,\infty),\vert\cdot\vert,\omega=N\sigma_Nr^{N-1}\mathcal L^1)$ so that
	\begin{align}\label{rearrangement}
		 {\sf m}\big(M_t(u)\big)=\omega(M_t(\hat{u})),\ \ \forall t>0,
	\end{align}
	where
	\begin{align*}
		M_t(u)=\lbrace x\in X: u(x)>t\rbrace\;\text{ and }\;M_t(\hat{u})=\lbrace y\in[0,\infty):\hat{u}(y)>t\rbrace,
	\end{align*}
whenever ${\sf m}\big(M_t(u)\big)<\infty$ for every $t>0$. 
In particular, the open set $\Omega\subset X$ can be rearranged into the interval $\Omega^*=[0,r]$ with $ {\sf m}(\Omega)=\omega([0,r])$, i.e., $ {\sf m}(\Omega)=\sigma_N r^N,$ with the convention that $\Omega^*=[0,\infty)$ if $ {\sf m}(\Omega)=+\infty$.
Using the isoperimetric inequality \eqref{eqn-isoperimetric-2} and the co-area formula (see e.g. Miranda \cite{Miranda}), Nobili and Violo proved recently \cite[Theorem 3.6]{NV} (see also \cite{NV2}) that the following \textit{P\'olya--Szeg\H o inequality} holds for any $u\in W^{ 1,p}(X,d,{\sf m})$:
\begin{equation}\label{Polya-Szego}
	{\sf Ch}_p(u)=\int_X |\nabla u|^pd {\sf m}\geq {\sf AVR}_ {\sf m}^\frac{p}{N}\int_0^\infty|\hat u'(r)|^p d\omega(r).
\end{equation}

	\subsection{Sharp weighted log-Sobolev inequality on the $1$-dimensional model space} 	
	Let $E\subset\mathbb{R}^n$ be an open convex cone and $\omega\colon E\to(0,\infty)$ be a log-concave homogeneous weight of class ${C}^1$ and degree $\tau\ge 0$, i.e., $\omega(\lambda x)=\lambda^\tau\omega(x)$. 
	
	Let $p>1$ and consider  $
		W^{1,p}(E;\omega)=\lbrace u\in L^p(E;\omega):\nabla u\in L^p(E;\omega)\rbrace\,.
	$
According to a recent result of Balogh, Don and Krist\'aly \cite{BDK},  for every function $u\in W^{1,p}(E;\omega)$ for which $\displaystyle\int_E\vert u\vert^p\omega\,dx=1$, we have
	\begin{align}\label{thm 1.1 BDK}
		\mathcal{E}_{\omega,E}(\vert u\vert^p)\colon=\int_E\vert u\vert^p\log\vert u\vert^p\omega\,dx\le \frac{n+\tau}{p}\,\log\bigg(\mathcal{L}_{\omega,p}\int_E\vert\nabla u\vert^p\omega\,dx\bigg)
	\end{align}
	where
$$
		\mathcal{L}_{\omega,p}\colon=\frac{p}{n+\tau}\bigg(\frac{p-1}{e}\bigg)^{p-1}\bigg(\Gamma\bigg(\frac{n+\tau}{p'}+1\bigg)\int_{B\cap E}\omega\,dx\bigg)^{-\frac{p}{n+\tau}}\,,
$$
and $B$ is the unit ball with center at the origin in $\mathbb R^n$.   Moreover, when $\tau>0$, equality holds in \eqref{thm 1.1 BDK} if and only if  the extremal function belongs to the family of Gaussians
\begin{equation}\label{Gaussian}
	u_{\lambda,x_0}(x)=\lambda^\frac{n+\tau}{pp'}\left(\Gamma\left(\frac{n+\tau}{p'}+1\right)\int_{B\cap E}\omega\,dx\right)^{-\frac{1}{p}}e^{-\lambda\frac{|x+x_0|^{p'}}{p}},\ \  x\in E, \ \lambda>0,
\end{equation}
where $x_0\in -\partial E\cap \partial E$ and $\omega(x+x_0)=\omega(x)$ for every $x\in E$.

Let $N>1$ and consider the $1$-dimensional cone $E=[0,\infty)$ endowed with the usual Euclidean distance $\vert\cdot\vert$ and measure $\omega=N\sigma_Nr^{N-1}\mathcal L^1$. It is clear that $\omega$ is a log-concave weight of class $C^1$ with homogeneity $\tau=N-1>0$. In this particular case, \eqref{thm 1.1 BDK} implies that on the   model space  $([0,\infty),\vert\cdot\vert, \omega)$  the 
following log-Sobolev inequality holds: for every $v\in W^{1,p}([0,\infty);\omega)$ with 
$N\sigma_N\displaystyle \int_0^\infty\vert v(r)\vert^p r^{N-1}dr=1$, one has
\begin{equation}\label{log-Sob-1D}
	N\sigma_N\int_{0}^\infty\vert v(r)\vert^p\log\vert v(r)\vert^pr^{N-1}dr\le \frac{N}{p}\log\bigg(\mathcal{L}_{N,p}N\sigma_N\int_0^\infty\vert v'(r)\vert^pr^{N-1}dr\bigg),
\end{equation}
where
\begin{equation}	\label{explicit-L}
		\mathcal
L_{p,N}:=\mathcal{L}_{\omega,p}={\small
	\frac{p}{N}\left(\frac{p-1}{e}\right)^{p-1}\left(\sigma_N{\Gamma\left(\frac{N}{p'}+1\right)}\right)^{-\frac{p}{N}}}. 
	\end{equation}


\section{Sharp log-Sobolev inequality in  ${\sf CD}(0,N)$ spaces: proof of Theorem \ref{log-Sobolev-main}} \label{section3}

\subsection{Proof of the log-Sobolev inequality \eqref{LSI}} Fix any function 
$ u\in W^{ 1,p}(X,d,{\sf m})$  with the property $\displaystyle\int_X |u|^pd  {\sf m}=1$. Without loss of generality, we may assume that $u\geq 0$, since $|\nabla|u||\leq |\nabla u|.$ 

On the one hand,  one can show that the operation of rearrangement described in \S \ref{subsection212} leaves both the $L^p$-norm  and entropy invariant,  i.e.,  we have that $$1=\displaystyle\int_X u^pd  {\sf m}=N\sigma_N\int_0^\infty\hat u(r)^pr^{N-1}dr,$$
and 
$$\int_{X}u^p\log u^pd {\sf m}=N\sigma_N\int_0^\infty\hat u(r)^p\log \hat u(r)^pr^{N-1}dr.$$
The proofs are based on the layer cake representation, see e.g. Lieb and Loss \cite{LL}, combined with   \eqref{rearrangement}; for the entropy term we use in addition the 'signed measure' described in \cite[p. 27]{LL} in order to split in a suitable way the function $s\mapsto s^p\log s^p$, $s>0$.

 On the other hand, by the P\'olya--Szeg\H o inequality 
\eqref{Polya-Szego}, we obtain
$$\int_X |\nabla u|^pd {\sf m}\geq {\sf AVR}_ {\sf m}^\frac{p}{N}N\sigma_N\int_0^\infty|\hat u'(r)|^p r^{N-1}dr.$$
The required inequality \eqref{log-Sobolev-main} follows from the latter relations combined with inequality \eqref{log-Sob-1D}.

\subsection{Sharpness in the log-Sobolev inequality \eqref{LSI}}

The sharpness of \eqref{LSI} follows by the next result. 


\begin{proposition}\label{propo-1}
Let $p,N>1$ and   $(X,d, {\sf m})$ be a ${\sf CD}(0,N)$ space. Assume that for every $ u\in W^{ 1,p}(X,d,{\sf m})\setminus \{0\}$ one has
\begin{equation}\label{LSI-C-2-00}
	\displaystyle	\frac{\displaystyle\int_{X}|u|^p\log |u|^pd {\sf m}}{\displaystyle\int_X |u|^pd  {\sf m}}-\log\left(\displaystyle\int_X |u|^pd  {\sf m}\right)\leq \frac{N}{p}\log\left(C\frac{\displaystyle\int_X |\nabla u|^pd {\sf m}}{\displaystyle\int_X |u|^pd  {\sf m}} \right)
\end{equation}
for some $C>0$. Then 
	\begin{equation}\label{local-prop-0}
	\mathcal
	L_{p,N}\leq {{\sf AVR}_ {\sf m}^\frac{p}{N}}C.
\end{equation}	
\end{proposition} 

{\it Proof.} Fix any point $x_0\in X,$ and 
let $u(x):=\tilde u(d_{x_0}(x))$ in \eqref{LSI-C-2-00} for any nonnegative  $\tilde  u\in {\rm Lip}_c([0,\infty))$, where $d_{x_0}(x)=d(x_0,x)$, $x\in X$. [Note that $\tilde u(0)$ need not be $0$.] By using the  non-smooth
chain rule and the eikonal equation \eqref{eikonal}, it turns out that 
\begin{equation}\label{LSI-C-3-0}
	\displaystyle	\frac{\displaystyle\int_{X}\tilde u^p(d_{x_0}(x))\log \tilde u^p(d_{x_0}(x))d {\sf m}(x)}{\displaystyle\int_X \tilde u^p(d_{x_0}(x))d  {\sf m}(x)}-\log\left(\displaystyle\int_X \tilde u^p(d_{x_0}(x))d  {\sf m}(x)\right)\leq \frac{N}{p}\log\left(C\frac{\displaystyle\int_X |\tilde u'(d_{x_0}(x))|^pd {\sf m}(x)}{\displaystyle\int_X \tilde u^p(d_{x_0}(x))d  {\sf m}(x)} \right). 
\end{equation}
If we consider the push-forward measure $\nu={d_{x_0}}_\# {\sf m}$ of $ {\sf m}$ on $[0,\infty)$, a change of variable in \eqref{LSI-C-3-0} yields that for every nonnegative  $\tilde  u\in {\rm Lip}_c([0,\infty))\setminus \{0\}$, one has
\begin{equation}\label{LSI-C-3-1}
	\displaystyle	\frac{\displaystyle\int_0^\infty\tilde u^p(r)\log \tilde u^p(r)d\nu(r)}{\displaystyle\int_0^\infty \tilde u^p(r)d \nu(r)}-\log\left(\displaystyle\int_0^\infty \tilde u^p(r)d \nu(r)\right)\leq \frac{N}{p}\log\left(C\frac{\displaystyle\int_0^\infty |\tilde u'(r)|^pd\nu(r)}{\displaystyle\int_0^\infty \tilde u^p(r)d \nu(r)} \right). 
\end{equation}

Now, we want to `replace' the measure $\nu$ by $\omega$ in \eqref{LSI-C-3-1}, where $\omega=N\sigma_N r^{N-1}\mathcal L^1$, $r>0$.
For simplicity of notations, we consider for every $r>0$,  
\begin{equation}\label{theta-r-N}
	\theta_{N,r}^+=\frac{ {\sf m}^+( B(x_0,r))}{N\sigma_N r^{N-1}},
\end{equation}
where 
$ {\sf m}^+( B(x_0,r))$ is the Minkowski content of $B(x_0,r)$, see \eqref{minkowski-content}. 
We are going to prove that
\begin{equation}\label{RN-relation}
	\frac{d\nu}{d\omega}(r)=\theta_{N,r}^+\ \ {\rm for\ a.e.}\ r>0,	
\end{equation} 
where $\frac{d\nu}{d\omega}$ stands for the 
Radon--Nikodym derivative.  

We first prove that the push-forward measure $\nu={d_{x_0}}_\# {\sf m}$ is absolutely continuous with respect to the measure $\omega=N\sigma_N r^{N-1}\mathcal L^1$ on $[0,\infty)$. To see this, we are going to prove that for every  $A\subset [0,\infty)$ such that $\omega(A)=\mathcal L^1(A)=0$, one has $\nu(A)=0,$ i.e., $ {\sf m}(\hat A)=0,$ where $\hat A=\cup_{r\in A}\partial B(x_0,r).$ Without loss of generality, we may assume that $A$ is compact and $0\notin A$. Fix $\varepsilon>0$ arbitrarily small; we show that $ {\sf m}(\hat A)<\varepsilon.$ Note that by Bishop--Gromov theorem, the function $\phi(r)=\frac{ {\sf m}(B(x_0,r))}{r^N}$ is non-increasing; if $r_0=\min\{r:r\in A\}$ and  $r_1=\max\{r:r\in A\}$, then 
$M:=\max_{r\in A} \phi(r)=\phi(r_0)$ and $K:=N \max_{r\in A} r^{N-1}=Nr_1^{N-1}$. Let $\delta=\frac{\varepsilon}{KM}>0.$ Since $\mathcal L^1(A)=0$,  we can find a finite covering of $A$ by disjoint intervals $\{(a_i,b_i)\}_{i=1,...,L}$ such that
$\sum_{i=1}^L (b_i-a_i)<\delta.$ Due to the fact that $A\subset \cup_{i=1}^L(a_i,b_i)$, one has that 
$\hat A\subset \cup_{i=1}^L (B(x_0,b_i)\setminus B(x_0,a_i))$. Therefore, by using the above monotonicity properties and the mean-value theorem, it follows that 
\begin{eqnarray*}
	{\sf m}(\hat A)&\leq& \sum_{i=1}^L( {\sf m}(B(x_0,b_i))- {\sf m}(B(x_0,a_i)))=\sum_{i=1}^L(\phi(b_i)b_i^N-\phi(a_i)a_i^N)\\&\leq& \sum_{i=1}^L\phi(a_i)(b_i^N-a_i^N)\leq MK\sum_{i=1}^L(b_i-a_i)<MK\delta=\varepsilon. 
\end{eqnarray*}
The arbitrariness of $\varepsilon>0$ implies that $ {\sf m}(\hat A)=0,$ which concludes the proof of the absolutely continuity of $\nu$  with respect to the measure $\omega$. 

 In order to prove \eqref{RN-relation}, we shall use the result of Sturm \cite[Theorem 2.3]{Sturm-2},  according to which
\begin{equation}\label{N-L}
	{\sf m}(B(x_0,r))=\int_0^r {\sf m}^+( B(x_0,s))ds, \ r>0,
\end{equation}
combined with the fact that $r\mapsto \theta_{N,r}^+$ is non-increasing on $(0,\infty).$
This implies in particular,  that the function $r \to  {\sf m}^+( B(x_0,r)) $ is continuous for almost every $r >0$. At the point of continuity of this function, we have that its integral $r \to  {\sf m}(B(x_0,r))$ is differentiable and its derivative is equal to $ {\sf m}^+( B(x_0,r))$. Moreover, by the definition of the $\nu={d_{x_0}}_\# {\sf m}$ we have that
\begin{equation}\label{relation-mu-nu}
	{\sf m}(B(x_0,r))=\nu([0,r)),\ r>0.
\end{equation} 
Let us fix $r>0$ such that the function $t\to  {\sf m}(B(x_0, t))$  is differentiable at $r$ and such that $$\frac{d {\sf m}(B(x_0,r))}{dr}(r) =  {\sf m}^+( B(x_0,r)),$$ which is of full measure in $(0,\infty)$ due to the above consideration; thus, by \eqref{N-L} and \eqref{relation-mu-nu}, one has
$$\frac{d\nu}{d\omega}(r)=\lim_{t\to 0}\frac{\nu((r-t,r+t))}{\omega((r-t,r+t))}=\lim_{t\to 0}\frac{ {\sf m}(B(x_0,r+t))- {\sf m}(B(x_0,r-t))}{\sigma_N ((r+t)^N-(r-t)^N)}=\frac{ {\sf m}^+( B(x_0,r))}{N\sigma_N r^{N-1}}=\theta_{N,r}^+,$$
which is precisely relation \eqref{RN-relation}. 

By \eqref{LSI-C-3-1} and \eqref{RN-relation} one has  
$$
	\displaystyle	\frac{\displaystyle\int_0^\infty\tilde u^p(r)\log \tilde u^p(r)\theta_{N,r}^+d\omega(r)}{\displaystyle\int_0^\infty \tilde u^p(r)\theta_{N,r}^+d\omega(r)}-\log\left(\displaystyle\int_0^\infty \tilde u^p(r)\theta_{N,r}^+d\omega(r)\right)\leq \frac{N}{p}\log\left(C\frac{\displaystyle\int_0^\infty |\tilde u'(r)|^p\theta_{N,r}^+d\omega(r)}{\displaystyle\int_0^\infty \tilde u^p(r) \theta_{N,r}^+d\omega(r)} \right). 
$$
If we insert $\tilde u(r)=v(\lambda r)\geq 0$ for every $\lambda>0$ (with $v\in {\rm Lip}_c([0,\infty))\setminus \{0\}$ and  $0\leq v\leq 1$) into the latter inequality, by the scaling invariance of the log-Sobolev inequality and a change of variables imply that 
\begin{equation}\label{eq-new-v}
	\displaystyle	\frac{\displaystyle\int_0^\infty v^p(s)\log v^p(s)\theta_{N,\frac{s}{\lambda}}^+d\omega(s)}{\displaystyle\int_0^\infty v^p(s)\theta_{N,\frac{s}{\lambda}}^+d\omega(s)}-\log\left(\displaystyle\int_0^\infty v^p(s)\theta_{N,\frac{s}{\lambda}}^+d\omega(s)\right)\leq \frac{N}{p}\log\left(C\frac{\displaystyle\int_0^\infty |v'(s)|^p\theta_{N,\frac{s}{\lambda}}^+d\omega(s)}{\displaystyle\int_0^\infty v^p(s) \theta_{N,\frac{s}{\lambda}}^+d\omega(s)} \right).
\end{equation}
 Now, we prove that 
 \begin{equation}\label{AVR-limit}
 	{\sf AVR}_ {\sf m}=\lim_{r\to \infty}\theta_{N,r}^+.
 \end{equation}
Clearly, the limit exists, due to \eqref{theta-r-N} and the Bishop--Gromov monotonicity property. First, for every $R>0$ one has by \eqref{N-L} that 
$$\frac{	{\sf m}(B(x_0,R))}{\sigma_NR^N}=\frac{1}{\sigma_NR^N}\int_0^R {\sf m}^+( B(x_0,s))ds=\frac{N}{R^N}\int_0^R  s^{N-1}\theta_{N,s}^+ds\geq \theta_{N,R}^+.
$$
If $R\to \infty$, it turns out that ${\sf AVR}_ {\sf m}\geq \lim_{r\to \infty}\theta_{N,r}^+.$
Conversely, if $R>r$,  we have that 
$$\frac{	{\sf m}(B(x_0,R))-	{\sf m}(B(x_0,r))}{\sigma_NR^N}=\frac{1}{\sigma_NR^N}\int_r^R {\sf m}^+( B(x_0,s))ds=\frac{N}{R^N}\int_r^R  s^{N-1}\theta_{N,s}^+ds\leq \frac{R^N-r^N}{R^N}\theta_{N,r}^+.
$$
Letting first $R\to \infty$ and then $r\to \infty$, it follows that ${\sf AVR}_ {\sf m}\leq \lim_{r\to \infty}\theta_{N,r}^+,$ which concludes the proof of \eqref{AVR-limit}. 

Letting $\lambda\to 0^+$ in \eqref{eq-new-v}, the limit in \eqref{AVR-limit} implies that 
\begin{equation}\label{eq-new-vvv}
	\displaystyle	\frac{\displaystyle\int_0^\infty v^p(s)\log v^p(s)d\omega(s)}{\displaystyle\int_0^\infty v^p(s)d\omega(s)}-\log\left({\sf AVR}_ {\sf m}\displaystyle\int_0^\infty v^p(s) d\omega(s)\right)\leq \frac{N}{p}\log\left(C\frac{\displaystyle\int_0^\infty |v'(s)|^pd\omega(s)}{\displaystyle\int_0^\infty v^p(s) d\omega(s)} \right).
\end{equation}
Since the Gaussian function $v(r)=e^{-\frac{r^{p'}}{p}},\ r>0,$ can be approximated by Lipschitz functions with compact support in $[0,\infty)$, we can use it as a test function in \eqref{eq-new-vvv}; now, for this particular choice of $v$ in \eqref{eq-new-vvv},  a straightforward computation yields that 
$$-\frac{N}{p'}-\log\left({\sf AVR}_ {\sf m}\sigma_N \Gamma\left(\frac{N}{p'}+1\right)\right)\leq \frac{N}{p}\log\left(C\left(\frac{p'}{p}\right)^p\frac{N}{p'} \right).$$
After a reorganization of the above terms, it follows that
$\mathcal
L_{p,N}\leq {{\sf AVR}_ {\sf m}^\frac{p}{N}}C,$ which is precisely the required relation \eqref{local-prop-0}.

\hfill $\square$

Some comments are in order.

\begin{remark}\rm \label{remark-equiv}
		A closer look at the two statements of  Theorem \ref{log-Sobolev-main} gives the following result: Given $p,N>1$ and $(X,d, {\sf m})$ a ${\sf CD}(0,N)$ space, the following two statements are equivalent: 
	\begin{itemize}
		\item[(i)] $(X,d, {\sf m})$ supports the $L^p$-log-Sobolev inequality, i.e., there exists $C>0$ such that 
		for every  
		$ u\in W^{ 1,p}(X,d,{\sf m})$  with $\displaystyle\int_X |u|^pd  {\sf m}=1$ one has that 
		\begin{equation} \label{LSI-CC}
		\int_{X}|u|^p\log |u|^pd {\sf m}\leq \frac{N}{p}\log\left(C\int_X |\nabla u|^pd {\sf m} \right); 
		\end{equation}
		\item[(ii)] ${\sf AVR}_ {\sf m}>0.$
	\end{itemize}
	We should mention  that this consideration might not come as a surprise; indeed, a similar equivalence is well-known in the Riemannian context for Sobolev-type inequalities, see  Coulhon and Saloff-Coste \cite{C-SC},  	do Carmo and Xia \cite{doCarmo-Xia}, and Hebey \cite{Hebey}. 
	
	Concerning the proof, we observe that the implication  (ii)$\implies$(i) is precisely the first statement of Theorem \ref{log-Sobolev-main} with $C=\mathcal
	L_{p,N}{\sf AVR}_ {\sf m}^{-\frac{p}{N}}$, while (i)$\implies$(ii) follows from its second statement, written equivalently as 
		$
	 {\sf AVR}_{\sf m} \geq \left( \frac{\mathcal L_{p, N}}{C}\right) ^{\frac{N}{p}}.
	 $
	 With  these observations and by introducing the notation for the optimal log-Sobolev constant 
	$$\mathcal{C}_{LS} := \inf \left\{ C>0: \text{\eqref{LSI-CC} holds for all} \  u\in W^{ 1,p}(X,d,{\sf m})\ \text{ with} \  \displaystyle\int_X |u|^pd  {\sf m}=1 \right\},$$
	we can reformulate the statement of Theorem \ref{log-Sobolev-main} simply into
$
 {\sf AVR}_ {\sf m} = \left( \frac{\mathcal L_{p, N}} {\mathcal{C}_{LS} }\right) ^{\frac{N}{p}}.
$
\end{remark}

\begin{remark}\rm 
	The proof of Proposition \ref{propo-1} is elementary on Riemannian manifolds. Indeed, if $(M,g)$ is an $n$-dimensional  complete Riemannian manifold endowed with its canonical measure $ {\sf m}=dv_g$,  for every $\varepsilon>0$, there exists a local chart
	$(\Omega,\phi)$ of $M$ at the point $x_0\in M$ and a number $\delta>0$
	such that $\phi(\Omega)=B(0,\delta)\subset \mathbb R^n$, and the components $g_{ij}$
	of the metric $g$ satisfy
	\begin{equation}\label{two-sided}
		(1-\varepsilon)\delta_{ij}\leq g_{ij} \leq (1+\varepsilon)\delta_{ij},\ \ i,j=1,...,n,
	\end{equation}
	in the sense of bilinear forms; here, $B(0,\delta)$ is the
	$n-$dimensional Euclidean ball of center $0$ and radius $\delta>0$, see e.g. Hebey \cite{Hebey}. Now, by using \eqref{two-sided}, we can `replace' the measure $ {\sf m}=dv_g$ with the usual Lebesgue measure $\mathcal L^n$ in $\mathbb R^n$, and the use of the Gaussian functions (which are the extremals in the Euclidean log-Sobolev inequality) provides the required estimate. Clearly, such a strategy is not applicable in the nonsmooth setting of ${\sf CD}(0,N)$ spaces, due to the lack of local charts and the two-sided estimate \eqref{two-sided}. 
\end{remark}

\begin{remark}\rm \label{remark-Ni}
	Theorem \ref{log-Sobolev-main} provides generalisation with a simple proof of the celebrated rigidity result of Ni \cite{Ni} for the case $p=2$: 
	
	\textit{Let $(M,g)$ be a complete $n$-dimensional Riemannian manifold with nonnegative Ricci curvature, and $p>1$.  Then 
 the log-Sobolev inequality 
	\begin{equation}\label{LSI-Li}
		\int_{M}|u|^p\log |u|^pd v_g\leq \frac{n}{p}\log\left(\mathcal
		L_{p,n}\int_M |\nabla_g u|^pd v_g \right) 
	\end{equation}
	 holds for every  
	 $ u\in W^{ 1,p}(M)$ with $\displaystyle\int_M |u|^pd  v_g=1$ if and only if $(M,g)$ is isometric to $\mathbb R^n$.} \\
 \noindent Indeed, if \eqref{LSI-Li} holds, by the sharpness of Theorem \ref{log-Sobolev-main} we necessarily have that $\mathcal
L_{p,n}{\sf AVR}_ {dv_g}^{-\frac{p}{n}}\leq \mathcal
L_{p,n}$, i.e., ${\sf AVR}_ {dv_g}\geq 1.$ But, by definition, we always have that  ${\sf AVR}_ {dv_g}\leq 1,$ thus ${\sf AVR}_ {dv_g}= 1,$ which implies that $(M,g)$ is isometric to $\mathbb R^n$, see  Petersen \cite[Exercise 7.5.10]{Petersen}.
\end{remark}


	\section{Sharp hypercontractivity estimates in ${\sf CD}(0,N)$ spaces: proof of Theorem \ref{Hopf--Lax-theorem}}\label{section4}

Before  the proof of \eqref{hyperc-estimate}, we are going to provide some important properties of the Hopf--Lax semigroup. In the sequel, we assume the assumptions of Theorem \ref{Hopf--Lax-theorem} are satisfied. 

\subsection{Hamilton--Jacobi inequality of the Hopf--Lax semigroup}
 The first proposition and its proof are based on ideas developed in \cite{AGS} and \cite{GRS}. 
  
\begin{proposition}\label{proposition-hyper}
	Let $p>1$,  $t_0>0$, $x_0\in X$ and  $u\in \mathcal  F_{t_0,x_0}(X)$ be fixed. Then the following statements hold$:$ 
	\begin{itemize}
		\item[(i)] 	$t_*(x):=\sup\{t>0:{\bf Q}_{t}u(x)>-\infty\}\geq t_0$ for every $x\in X;$
		\item[(ii)] $(t,x)\mapsto {\bf Q}_tu(x)$ is locally Lipschitz on $(0,t_0)\times X;$
		\item[(iii)] for every $(t,x)\in (0,t_0)\times X$ one has the Hamilton--Jacobi inequality
		\begin{equation}\label{HJ-inequality}
				\frac{d^+}{dt}{\bf Q}_tu(x)\le -\frac{\vert D{\bf Q}_tu\vert^p(x)}{p},
		\end{equation}
	where $\frac{d^+}{dt}$ stands for the right derivative$; $
		\item[(iv)] for every $x\in  X$ one has 	\begin{align}\label{ineq claim lemma 4.1}
			\frac{d^+}{dt}\bigg\vert_{t=0}{\bf Q}_tu(x)\ge -\frac{\vert Du\vert^p(x)}{p}.
		\end{align}

	\end{itemize}
\end{proposition}

{\it Proof.}
	(i)\&(ii)  Let us fix arbitrarily $0<t_1<t_2<t_0$ and the compact set $K\subset X$.  In order to prove the claims, it is enough to show that the function $(t,x) \mapsto {\bf Q}_{t}u(x)$  is well-defined and Lipschitz continuous on $[t_1,t_2]\times K$. Indeed,  one has for every $(t,x)\in [t_1,t_2]\times K$ that
		\begin{eqnarray*}
		{\bf Q}_{t}u(x)&=&\inf_{y\in X}\left\{u(y)+\frac{d(x,y)^{p'}}{p't^{p'-1}}\right\}\\&\geq &\inf_{y\in X}\left\{u(y)+\frac{d(x_0,y)^{p'}}{p't_0^{p'-1}}\right\}+\inf_{y\in X}\left\{\frac{d(x,y)^{p'}}{p't^{p'-1}}-\frac{d(x_0,y)^{p'}}{p't_0^{p'-1}}\right\}\\&=&
		{\bf Q}_{t_0}u(x_0)+\inf_{y\in X}\left\{\frac{d(x,y)^{p'}}{p't^{p'-1}}-\frac{d(x_0,y)^{p'}}{p't_0^{p'-1}}\right\} >-\infty,
	\end{eqnarray*}
	uniformly in $x\in K$; here, we used the assumption (A2), i.e.,  ${\bf Q}_{t_0}u(x_0)>-\infty$, combined with the fact that 
	$$\lim_{d(x_0,y)\to  \infty}\left\{\frac{d(x,y)^{p'}}{p't^{p'-1}}-\frac{d(x_0,y)^{p'}}{p't_0^{p'-1}}\right\}=+\infty,$$
	which follows from  $t<t_2<t_0$ and the  triangle inequality. In particular, by the arbitrariness of $t\in [t_1,t_2]\subset (0,t_0)$, we have that $t_*(x)\geq t_0$ for every $x\in X,$ which proves (i).
	
	Moreover, we can find $R=R_K>0$ such that for every $(t,x)\in [t_1,t_2]\times K$, one has $${\bf Q}_{t}u(x)=\min_{y\in  \overline{B(x_0,R)}}\left\{u(y)+\frac{d(x,y)^{p'}}{p't^{p'-1}}\right\}.$$
	We observe that the function $[t_1,t_2]\times K\ni (t,x)\mapsto u(y)+\frac{d(x,y)^{p'}}{p't^{p'-1}}$ is uniformly Lipschitz in $y$, thus $[t_1,t_2]\times K\ni (t,x)\mapsto {\bf Q}_{t}u(x)$ is also  Lipschitz,  as the infimum of a family of uniformly Lipschitz functions. 
	
	(iii) The result of Ambrosio, Gigli and Savar\'e \cite[Theorem 3.5]{AGS} states that the Hamilton--Jacobi inequality \eqref{HJ-inequality} is valid for $p=2$, and for every $x\in X$ and $t\in (0,t_*(x));$ in particular, by (i) the inequality \eqref{HJ-inequality} is obtained for every $(t,x)\in (0,t_0)\times X$. The generic case $p\neq 2$ follows in a similar way as in \cite{AGS} after an almost trivial adaptation,   either from  Ambrosio, Gigli and Savar\'e \cite{AGS-2} or Gozlan,  Roberto and Samson \cite{GRS}.
	
	(iv) 
Let $x\in X$ be fixed, and consider a positive sequence $\lbrace t_n\rbrace_{n\in\mathbb{N}}$ with $t_n\to 0^+$ as $n\to \infty$; clearly, we may assume that $t_n<t_\#$ for every $n\in \mathbb N,$ where $t_\#<t_0$.  According to (i) (i.e., ${\bf Q}_{t}u(z)>-\infty$ for every $(t,z)\in (0,t_0)\times X$), there exist $y_n\in X$ $(n\in \mathbb N)$ and $y_\#\in X$  such that
	\begin{equation}\label{minimum-seq}
		{\bf Q}_{t_n}u(x)=u(y_n)+\frac{d^{p'}(x,y_n)}{p'\cdot t_n^{p'-1}}\ \ {\rm and}\ \  {\bf Q}_{t_\#}u(x)=u(y_\#)+\frac{d^{p'}(x,y_\#)}{p'\cdot t_\#^{p'-1}} .
	\end{equation}
	We are going to prove that the sequence $\{y_n\}$ is bounded. To see this, we have 
	\begin{eqnarray*}
	-\infty&<&	{\bf Q}_{t_\#}u(x)=u(y_\#)+\frac{d^{p'}(x,y_\#)}{p'\cdot t_\#^{p'-1}}\leq u(y_n)+\frac{d^{p'}(x,y_n)}{p'\cdot t_\#^{p'-1}}\\
	&=&u(y_n)+\frac{d^{p'}(x,y_n)}{p'\cdot t_n^{p'-1}}+\frac{d^{p'}(x,y_n)}{p'\cdot t_\#^{p'-1}}-\frac{d^{p'}(x,y_n)}{p'\cdot t_n^{p'-1}}={\bf Q}_{t_n}u(x)+\frac{d^{p'}(x,y_n)}{p'\cdot t_\#^{p'-1}}-\frac{d^{p'}(x,y_n)}{p'\cdot t_n^{p'-1}}\\&\leq & u(x)+ \frac{1}{p'}d^{p'}(x,y_n)\left\{\frac{1}{t_\#^{p'-1}}-\frac{1}{t_n^{p'-1}}\right\}.
	\end{eqnarray*}
If $d(x,y_n)\to \infty$ as $n\to \infty$, since $t_n<t_\#$ for every $n\in \mathbb N,$ the latter expression tends to $-\infty$, which contradicts the fact $-\infty<	{\bf Q}_{t_\#}u(x).$ Therefore, there exists $R>0$ such that $\{y_n\}\subset B(x_0,R)$ for every $n\in \mathbb R.$
	
	By \eqref{minimum-seq} and ${\bf Q}_{t_n}u\le u$, one has
	\begin{align*}
		\frac{d^{p'}(x,y_n)}{p'\cdot t_n^{p'-1}}\le u(x)-u(y_n),\ n\in \mathbb N. 
	\end{align*}
	Since $u\in  {\rm Lip}_{\rm loc}(X)$, it  turns out that $u$ is globally Lipschitz on the compact set $\overline {B(x_0,R+d(x_0,x))}$ and $x,y_n\in \overline {B(x_0,R+d(x_0,x))}$, $n\in \mathbb N;$ in particular, there exists $L>0$ such that   $u(x)-u(y_n)\le L\cdot d(x,y_n)$ for every $n\in \mathbb N$. Rearranging the previous inequality then yields that 
	\begin{align*}
		d(x,y_n)\le \big(p'\cdot L\big)^{1/p'-1}\cdot t_n,\, \ n\in \mathbb N,
	\end{align*}
	and since $t_n\to 0$, we also have that $d(x,y_n)\to 0$ as $n\to \infty$. By the definition of the slope $|Du|$,  for a fixed $\varepsilon>0$ there exists $N_\varepsilon\in\mathbb{N}$ such that
	\begin{align}\label{ineq with epsilon}
		\frac{u(y_n)-u(x)}{d(x,y_n)}\ge -\vert Du\vert(x)-\varepsilon\ ,\ \forall n\ge N_\varepsilon.
	\end{align}
	Therefore, by applying \eqref{minimum-seq} and \eqref{ineq with epsilon}, we have that
	\begin{eqnarray*}
		\frac{{\bf Q}_{t_n}u(x)-u(x)}{t_n}&=&\frac{u(y_n)-u(x)}{t_n}+\frac{d^{p'}(x,y_n)}{p'\cdot t_n^{p'}}\\&=&\frac{u(y_n)-u(x)}{d(x,y_n)}\cdot\frac{d(x,y_n)}{t_n}+\frac{1}{p'}\bigg(\frac{d(x,y_n,)}{t_n}\bigg)^{p'}\\&\geq &
		 -\big(\vert Du\vert(x)+\varepsilon\big)\cdot\frac{d(x,y_n)}{t_n}+\frac{1}{p'}\bigg(\frac{d(x,y_n,)}{t_n}\bigg)^{p'}\\&\ge&-\frac{\big(\vert Du\vert(x)+\varepsilon)^p}{p}\,,
	\end{eqnarray*}
		where  the last estimate  follows by  Young's inequality, written in the following form:
	\begin{align*}
		-As+\frac{1}{p'}\cdot s^{p'}\ge -\frac{A^p}{p}\ ,\ \forall A,s>0\,.
	\end{align*}
	Finally, since $t_n\to 0 $ as $n\to\infty$, combined with  the arbitrariness of $\varepsilon>0$, we obtain the desired inequality \eqref{ineq claim lemma 4.1}. 
\hfill $\square$

\subsection{Sobolev regularity of the Hopf--Lax semigroup}

In the sequel, besides the locally Lipschitz property of $(t,x)\mapsto {\bf Q}_tu(x)$  on $(0,t_0)\times X$ (cf. Proposition \ref{proposition-hyper}/(ii)), we prove an important regularity property of an integral  involving  $e^{{\bf Q}_tu} $. Before doing this, we state an elementary inequality (whose proof is left to the reader), which plays a crucial role in our argument; namely, for every  $x,y,z\in X$, $t,s>0$ and  $r>1$,  one has
\begin{equation}\label{metric-relation}
	\frac{d^r(x,y)}{t^{r-1}}+\frac{d^r(y,z)}{s^{r-1}}\geq \frac{d^r(x,z)}{(t+s)^{r-1}}. 
\end{equation}

\begin{proposition}\label{proposition-hyper-2}
		Let $p>1$, $t_0>0$, $x_0\in X$ and  $0<\alpha\leq \beta$ be fixed, and let $q:[0,\tilde t]\to [\alpha,\beta]$ be a $C^1$-function  on $(0,\tilde t)$ for some $\tilde t\in (0,t_0).$ 
		Then for every   $u\in \mathcal F_{t_0,x_0}(X)$ with $(1+d^{p'}(x_0,\cdot))e^{\alpha u}\in  L^1(X,{\sf m})$,  we have that

	\begin{itemize}
		\item[(i)] 
		$(1+d^{p'}(x_0,\cdot))e^{\beta u}\in  L^1(X,{\sf m});$
		\item[(ii)] the function 
		$$F(t)=\int_X  e^{q(t){\bf Q}_{t}u(x)}d{\sf m}(x)$$
		is well-defined and locally Lipschitz in $(0,\tilde t);$
		\item[(iii)] the function $x\mapsto e^{\frac{q(t)}{p}{\bf Q}_{t}u(x)}$ belongs to $ W^{ 1,p}(X,d,{\sf m})$ 
		for every 
		 $t\in (0,\tilde t)$.   
	\end{itemize}
\end{proposition}

\begin{proof}
	(i) Due to assumption (A1), the set $S:=u^{-1}([0,\infty))\subset X$ is bounded, i.e., there exists $R>0$ such that $S\subset B(x_0,R)$. In particular, ${\sf m}(S)<+\infty$ (see Sturm \cite[Theorem 2.3]{Sturm-2}) and $U:=\sup_X u<+\infty$; thus  
	 $$\int_S (1+d^{p'}(x_0,x))e^{\beta u(x)}d{\sf m}(x)\leq (1+R^{p'}) e^{\beta U}{\sf m}(S)<+\infty.$$
	On the other hand, since $u(x)\leq 0$ for every $x\in X\setminus S$ and $(1+d^{p'}(x_0,\cdot))e^{\alpha u}\in  L^1(X,{\sf m})$, for every  $\beta\geq \alpha$ one has that
	 $$\int_{X\setminus S} (1+d^{p'}(x_0,x))e^{\beta u(x)}d{\sf m}(x)\leq \int_{X\setminus S} (1+d^{p'}(x_0,x))e^{\alpha u(x)}d{\sf m}(x)<+\infty.$$

(ii) We first show that $e^{q(t){\bf Q}_{t}u}\in L^1(X,{\sf m})$ for every $t\in (0,\tilde t)$, which implies the well-definiteness of the function $F$.  Indeed, by definition one has that 
	${\bf Q}_{t}u\leq u$ for every $t\in (0,\tilde t)$. Therefore, due to the fact that $q(t)\in [\alpha,\beta]$ for every $t\in [0,\tilde t]$, one has for every $(t,x)\in (0,\tilde t)\times X$ that 
	\begin{equation}\label{first-estimate}
		 e^{q(t){\bf Q}_{t}u(x)}\leq e^{\max\{\alpha {\bf Q}_{t}u(x),\beta {\bf Q}_{t}u(x)\}}=\max\{e^{\alpha {\bf Q}_{t}u(x)},e^{\beta {\bf Q}_{t}u(x)}\}\leq \max\{e^{\alpha u(x)},e^{\beta u(x)}\}.
	\end{equation}
Therefore, the assumption $(1+d^{p'}(x_0,\cdot))e^{\alpha u}\in  L^1(X,{\sf m})$, combined with property (i), yields that $e^{q(t){\bf Q}_{t}u}\in L^1(X,{\sf m})$ for every $t\in (0,\tilde t)$.

In order to prove that $F$ is locally Lipschitz on $(0,\tilde t)$, we are first going to show that there exists $C_1>0$ such that
\begin{equation}\label{Q-t-estimate}
	|{\bf Q}_{t}u(x)|\leq C_1(1+d^{p'}(x_0,x)),\ \ \forall (t,x)\in (0,\tilde t)\times X.
\end{equation}
On one hand, by assumption (A3) and inequality \eqref{metric-relation} (note that $C_0>p't_0^{p'-1}$ and $t_0>\tilde t$), we have for every $ (t,x)\in (0,\tilde t)\times X$ that
$$	{\bf Q}_{t}u(x)=\inf_{y\in X}\left\{u(y)+\frac{d(x,y)^{p'}}{p't^{p'-1}}\right\}\geq M+\inf_{y\in X}\left\{-\frac{d(x_0,y)^{p'}}{C_0}+\frac{d(x,y)^{p'}}{p't_0^{p'-1}}\right\}\geq M-\frac{d(x_0,x)^{p'}}{b},$$ 
where $b=\left(C_0^{p-1}-t_0(p')^{p-1}\right)^{p'-1}>0.$
 One the other hand, by definition we have ${\bf Q}_{t}u\leq u$ for every $t\in (0,\tilde t)$; therefore, combining the latter two estimates, for every $(t,x)\in (0,\tilde t)\times X$ one has 
 $$|{\bf Q}_{t}u(x)|\leq \max \left\{|u(x)|,|M|+\frac{d(x_0,x)^{p'}}{b}\right\}.$$
 Since $U:=\sup_X u<+\infty$ and by assumption (A3),  \eqref{Q-t-estimate} follows at once. 
 
 Now, let us fix $t_1,t_2\in  (0,\tilde t)$, say $t_1<t_2$. Recall from Proposition \ref{proposition-hyper}/(ii) that $\tau\mapsto {\bf Q}_{\tau}u(x)$ is locally Lipschitz.    We are going to show that there exists $C_2>0$ such that
 \begin{equation}\label{dQ-t-estimate}
 	\max\left\{|\xi|:\xi\in \partial^0{\bf Q}_{t}u(x) \right\}\leq C_2(1+d^{p'}(x_0,x)),\ \ \forall (t,x)\in (t_1,t_2)\times X,
 \end{equation}
where $\partial^0{\bf Q}_{t}u(x)$ stands for the Clarke generalized gradient of the locally Lipschitz function $\tau\mapsto {\bf Q}_{\tau}u(x)$ at the point $t\in (t_1,t_2)$, see Clarke \cite{Clarke}. In this particular setting, it turns out that 
\begin{equation}\label{generalized-gradient}
	|\xi|\leq \max\left\{{\bf Q}_{t}^0u(x)(t;-1),{\bf Q}_{t}^0u(x)(t;1)\right\},\ \ \forall \xi\in \partial^0{\bf Q}_{t}u(x),
\end{equation}
see Clarke \cite[pp. 25-27]{Clarke}, where $${\bf Q}_{t}^0u(x)(t;v)=\limsup_{\{s\to t; h\to 0^+\}}\frac{{\bf Q}_{s+hv}u(x)-{\bf Q}_{s}u(x)}{h},\ v\in \mathbb R,$$
denotes the Clarke generalized directional derivative of $\tau\mapsto {\bf Q}_{\tau}u(x)$ at $t\in (t_1,t_2)$ in direction $v\in \mathbb R.$ 
Note that since $t\mapsto {\bf Q}_{t}u$ is non-increasing, one has  
\begin{equation}
	{\bf Q}_{t}^0u(x)(t;1)\leq 0.
\end{equation}
In order to estimate the term ${\bf Q}_{t}^0u(x)(t;-1)$, for $s>0$ and $x\in X$ let $y_{s,x}\in X$ be such that
$${\bf Q}_{s}u(x)=u(y_{s,x})+\frac{d(x,y_{s,x})^{p'}}{p's^{p'-1}};$$
denote the set of such elements by $m(x,s)\subset X.$
Then, one has 
\begin{eqnarray}\label{estimate-Q-negative}
\nonumber	{\bf Q}_{t}^0u(x)(t;-1)&=&\limsup_{s\to t; h\to 0^+}\frac{{\bf Q}_{s-h}u(x)-{\bf Q}_{s}u(x)}{h}\\&\leq&\nonumber\limsup_{s\to t; h\to 0^+}\frac{u(y_{s,x})+\frac{d(x,y_{s,x})^{p'}}{p'(s-h)^{p'-1}}-\left(u(y_{s,x})+\frac{d(x,y_{s,x})^{p'}}{p's^{p'-1}}\right)}{h}\\&=&\nonumber\limsup_{s\to t; h\to 0^+}\frac{d(x,y_{s,x})^{p'}}{p'}\cdot\frac{s^{p'-1}-(s-h)^{p'-1}}{hs^{p'-1}(s-h)^{p'-1}} \\&\leq &\frac{1}{p t^{p'}}\sup_{\substack{s\in [t_1,t_2] \\ y_{s,x}\in m(s,x)}}{d(x,y_{s,x})^{p'}}.
\end{eqnarray}
It remains to estimate the last term in \eqref{estimate-Q-negative}. First, since $y_{s,x}\in m(s,x)$ we have for every $(s,x)\in [t_1,t_2]\times X$  that
$$u(y_{s,x})+\frac{d(x,y_{s,x})^{p'}}{p's^{p'-1}}={\bf Q}_{s}u(x)\leq u(x).$$
This relation together with assumption (A3) implies that for every $(s,x)\in [t_1,t_2]\times X$ one has
$$\frac{d(x,y_{s,x})^{p'}}{p't_2^{p'-1}}\leq \frac{d(x,y_{s,x})^{p'}}{p's^{p'-1}}\leq u(x)-u(y_{s,x})\leq u(x)-M+\frac{d(x_0,y_{s,x})^{p'}}{C_0}.$$
Moreover, as before, by the metric inequality \eqref{metric-relation} one has that 
$$\frac{d(x_0,y_{s,x})^{p'}}{C_0}\leq \frac{d(x,y_{s,x})^{p'}}{p't_0^{p'-1}}+\frac{d(x,x_0)^{p'}}{b},$$ where 
$b=\left(C_0^{p-1}-t_0(p')^{p-1}\right)^{p'-1}>0$. Combining the latter two inequalities yields
$$\frac{1}{p'}\left(\frac{1}{t_2^{p'-1}}-\frac{1}{t_0^{p'-1}}\right)d(x_0,y_{s,x})^{p'}\leq u(x)-M+\frac{d(x,x_0)^{p'}}{b}.$$
Since  $t_2<t_0$, the latter estimate with relations \eqref{generalized-gradient}-\eqref{estimate-Q-negative} provide the claimed inequality \eqref{dQ-t-estimate}. 

Now, let $t,s\in [t_1,t_2]\subset (0,\tilde t)$ be arbitrarily fixed (say $t<s$). According to Proposition \ref{proposition-hyper}/(ii), the function $\tau \mapsto e^{q(\tau){\bf Q}_{\tau}u(x)}$ is locally Lipschitz on $(0,t_0)$. By using Lebourg's mean value theorem for the latter function, see Clarke \cite[Theorem 2.3.7]{Clarke}, there exist  $\theta=\theta_x\in (t,s)$ and $\xi=\xi_x\in  \partial^0{\bf Q}_{\theta}u(x) $ such that 
$$e^{q(t){\bf Q}_{t}u(x)}-e^{q(s){\bf Q}_{s}u(x)}=e^{q(\theta){\bf Q}_{\theta}u(x)}(q'(\theta){\bf Q}_{\theta}u(x)+q(\theta)\xi)(t-s).$$
The latter relation together with estimates \eqref{first-estimate}-\eqref{dQ-t-estimate} yields the existence of a constant $C_3>0$ such that
\begin{eqnarray*}
|F(t)-F(s)|&\leq& \int_X\left|e^{q(t){\bf Q}_{t}u(x)}-e^{q(s){\bf Q}_{s}u(x)}\right|d{\sf m}(x)\\&\leq& C_3\int_X\max\{e^{\alpha u(x)},e^{\beta u(x)}\}(1+d^{p'}(x_0,x))d{\sf m}(x)\cdot |t-s|.
\end{eqnarray*}
Due to (i) and assumption $(1+d^{p'}(x_0,\cdot))e^{\alpha u}\in  L^1(X,{\sf m})$, we conclude by the latter estimate that $F$ is locally Lipschitz on $(0,t_0).$

(iii) The fact that $x\mapsto e^{\frac{q(t)}{p}{\bf Q}_{t}u(x)}$ belongs to $L^p(X,{\sf m})$ for every  $t\in (0,\tilde t)$   is equivalent to the well-definiteness of $F$ on $(0,\tilde t)$, see (ii). It remains to prove that $\left|\nabla \left(e^{\frac{q(t)}{p}{\bf Q}_{t}u}\right)\right|\in L^p(X,{\sf m})$ for every $t\in (0,\tilde t)$.  

According to (ii), 
  the functions 
	$$H_\alpha(t)=\int_Xe^{\alpha{\bf Q}_{t}u(x)}d{\sf m}(x)\ \ {\rm and}\ \  H_\beta(t)=\int_Xe^{\beta{\bf Q}_{t}u(x)}d{\sf m}(x)$$
	are locally Lipschitz on $(0,\tilde t)$ (consider $q(t)=\alpha$ and $q(t)=\beta$, respectively). 
Since both of these functions are non-increasing, the right derivatives of both functions $H_\alpha$ and $H_\beta$ exist at {\it every} point $t\in (0,\tilde t)$, i.e.,   $-\infty<\frac{d^+}{dt}H_\alpha(t)<\infty$ and $-\infty<\frac{d^+}{dt}H_\beta(t)<\infty$ for every $t\in (0,\tilde t)$. In addition, by the Hamilton--Jacobi inequality, see Proposition \ref{proposition-hyper}/(iii),   one has for every  $t\in (0,\tilde t)$ that
\begin{eqnarray}\label{HJ-slopes}
\nonumber	-\infty&<&\frac{d^+}{dt}H_\alpha(t)=\alpha\int_X\frac{d^+}{dt}{\bf Q}_{t}u(x) e^{\alpha{\bf Q}_{t}u(x)}d{\sf m}(x)\leq -\frac{\alpha}{p}\int_X|D {\bf Q}_{t}u|^p(x) e^{\alpha{\bf Q}_{t}u(x)}d{\sf m}(x).
\end{eqnarray}
Due to Proposition \ref{proposition-hyper}/(ii), one has that ${\bf Q}_{t}u\in {\rm Lip}_{\rm loc}(X)$ for every $t\in (0,t_0)$; in particular, by \eqref{Cheeger-equality} we have  $|\nabla {\bf Q}_{t}u|(x)= |D {\bf Q}_{t}u|(x)$ for every $t\in (0,t_0)$ and ${\sf m}$-a.e. $x\in X $. Therefore, the previous estimate implies that
	$$\int_X|\nabla {\bf Q}_{t}u|^p(x) e^{\alpha{\bf Q}_{t}u(x)}d{\sf m}(x)<+\infty, \ \ \forall t\in (0,\tilde t).$$
	In the same way, we also have 
	$$\int_X|\nabla  {\bf Q}_{t}u|^p(x) e^{\beta{\bf Q}_{t}u(x)}d{\sf m}(x)<+\infty, \ \ \forall t\in (0,\tilde t).$$
	By \eqref{first-estimate} and the above two relations,  we obtain that
	\begin{equation}\label{Sobolev-estimate}
			\int_X|\nabla  {\bf Q}_{t}u|^p(x) e^{q(t){\bf Q}_{t}u(x)}d{\sf m}(x)<+\infty, \ \ \forall t\in (0,\tilde t),
	\end{equation}
which is equivalent, by means of the non-smooth chain rule, to the fact that 
$$\int_X\left|\nabla  \left(e^{\frac{q(t)}{p}{\bf Q}_{t}u(x)}\right)\right|^p d{\sf m}(x)<+\infty, \ \ \forall t\in (0,\tilde t),$$
i.e., $e^{\frac{q(t)}{p}{\bf Q}_{t}u}\in  W^{ 1,p}(X,d,{\sf m})$ for every  $t\in (0,\tilde t)$,  
	which concludes the proof. 
\end{proof}

\subsection{Proof of the hypercontractivity estimate \eqref{hyperc-estimate}}\label{subsection-hypercontractivity}
In the case $\alpha=\beta$,  relation  \eqref{hyperc-estimate} turns out to be trivial; thus, let us fix  $\beta>\alpha>0$ and $u\in \mathcal F_{t_0,x_0}(X)$ with $(1+d^{p'}(x_0,\cdot))e^{\alpha u}\in  L^1(X,{\sf m})$ and  $e^\frac{\alpha u}{p}\in   W^{ 1,p}(X,d,{\sf m})$. Fix an arbitrary $\tilde t\in (0,t_0)$, and for every $t\in [0,\tilde t]$ consider the functions 
	\begin{align*}
		q(t)=\frac{\alpha\beta}{(\alpha-\beta)t/\tilde{t}+\beta}\ \ {\rm and}\ \ \tilde F(t):=F(t)^{1/q(t)}=\bigg(\int_Xe^{q(t){\bf Q}_tu }d {\sf m}\bigg)^{1/q(t)}=\Vert e^{{\bf Q}_tu}\Vert_{L^{q(t)}(X,{\sf m})}.
	\end{align*}
We observe  that $q(0)=\alpha$, $q(\tilde{t})=\beta$, and $q$ is of class $C^1$ in $(0,\tilde{t})$, with
$
	q'(t)=\frac{\alpha\beta}{[(\alpha-\beta)t/\tilde{t}+\beta]^2}\cdot\frac{\beta-\alpha}{\tilde{t}}>0\,;
$
thus, $t\mapsto q(t)$ is strictly increasing. 
Finally,  we consider the function 
\begin{align*}
	x\mapsto w_t(x):=\frac{e^{\frac{q(t)}{p}{\bf Q}_tu(x)  }}{\tilde F(t)^\frac{q(t)}{p}}.
\end{align*}
By Proposition \ref{proposition-hyper-2}/(iii), $w_t$ belongs to $W^{ 1,p}(X,d,{\sf m})$  for every 
$t\in (0,\tilde t)$ and 
\begin{align*}
	\int_Xw_t^p\,d {\sf m}=1\ ,\ \forall t\in (0,\tilde{t}). 
\end{align*}
According to Theorem \ref{log-Sobolev-main}, one has
\begin{equation}\label{LSI-applied-w_t}
	\int_Xw_t^p\log w_t^pd {\sf m}\le\frac{N}{p}\log\bigg(\mathcal{L}_{p,N}{\sf AVR}_ {\sf m}^{-\frac{p}{N}}\int_X\vert \nabla w_t\vert^p\,d {\sf m}\bigg)\,.
\end{equation}
By the non-smooth chain rule, we have 
\begin{equation}\label{chain-rule-D}
	\vert \nabla w_t(x)\vert=\frac{e^{{\frac{q(t)}{p}\bf Q}_tu(x) }}{\tilde F(t)^\frac{q(t)}{p}}\cdot\frac{q(t)}{p}\cdot\vert \nabla{\bf Q}_tu(x)\vert\,,
\end{equation}
and 
\begin{equation}\label{identity-transformed}
	\int_Xw_t^p\log w_t^pd {\sf m}=\frac{\mathcal{E}\big(e^{q(t){\bf Q}_tu }\big)}{\tilde F(t)^{q(t)}}-\log\big(\tilde F(t)^{q(t)}\big),
\end{equation}
where 
\begin{align*}
	\mathcal{E}\big(e^{q(t){\bf Q}_tu }\big)=\int_Xe^{q(t){\bf Q}_tu }\log\big(e^{q(t){\bf Q}_tu }\big)d {\sf m}=\int_Xe^{q(t){\bf Q}_tu }q(t){\bf Q}_tu \,d {\sf m}.
\end{align*}
By Proposition \ref{proposition-hyper-2}, it turns out that $\tilde F(t)=F(t)^{1/q(t)}$, and thus $\tilde F$ is a locally Lipschitz function on $(0,\tilde{t})$. In particular, its right derivative exists at every $t\in (0,\tilde t)$ and its expression is given by
\begin{align}\label{right-derivative}
\nonumber	\frac{d^+}{dt}\tilde F(t)=&\frac{d^+}{dt}e^{\log \tilde F(t)}=\frac{d^+}{dt}e^{\frac{1}{q(t)}\log\left(\displaystyle\int_Xe^{q(t){\bf Q}_tu}d {\sf m}\right)}\\=&\nonumber \tilde F(t)\left(-\frac{q'(t)}{q(t)^2}\log\left(\int_Xe^{q(t){\bf Q}_tu }d {\sf m}\right)+\frac{\displaystyle\int_Xe^{q(t){\bf Q}_tu }\left(q(t)\frac{d^+}{dt}{\bf Q}_tu +q'(t){\bf Q}_tu \right)d {\sf m}}{q(t)\cdot \tilde F(t)^{q(t)}}\right)\\=&\frac{\tilde F(t)}{q(t)^2}\bigg(-q'(t)\log\big(\tilde F(t)^{q(t)}\big)+\frac{q(t)^2}{\tilde F(t)^{q(t)}}\int_Xe^{q(t){\bf Q}_tu }\frac{d^+}{dt}{\bf Q}_tu\,d {\sf m}+\frac{q'(t)}{\tilde F(t)^{q(t)}}\mathcal{E}\big(e^{q(t){\bf Q}_tu }\big)\bigg).
\end{align}
Note that all three terms in the last parenthesis are well-defined for every $t\in (0,\tilde t)$. Indeed, the first term is well-defined due to Proposition \ref{proposition-hyper-2}/(ii), while the second and third terms are bounded from above (due to the monotonicity of $t\mapsto {\bf Q}_tu$ and the log-Sobolev inequality \eqref{LSI-applied-w_t} combined with the identity \eqref{identity-transformed}); if one of them were $-\infty$, this  would contradict the fact that $\frac{d^+}{dt}\tilde F(t)\in \mathbb R$, $t\in (0,\tilde t)$. 

By using the Hamilton--Jacobi inequality \eqref{HJ-inequality} together with  $|\nabla {\bf Q}_{t}u|(x)=|D {\bf Q}_{t}u|(x)$ for every $t\in (0,\tilde t)$ and ${\sf m}$-a.e. $x\in X $, see \eqref{Cheeger-equality}, we obtain for every $t\in (0,\tilde t)$ that 
\begin{align*}
	\frac{\frac{d^+}{dt}\tilde F(t)}{\tilde F(t)}\le\frac{q'(t)}{q(t)^2\tilde F(t)^{q(t)}}\bigg(\mathcal{E}\big(e^{q(t){\bf Q}_tu }\big)-\tilde F^{q(t)}\log\big(\tilde F(t)^{q(t)}\big)-\frac{q(t)^2}{q'(t)}\int_Xe^{q(t){\bf Q}_tu }\frac{\vert \nabla{\bf Q}_tu\vert^p}{p}\,d {\sf m}\bigg).
\end{align*}
Applying the elementary inequality $\log(ey)\leq y$ for every $y>0$, together with the log-Sobolev inequality \eqref{LSI-applied-w_t} and relations \eqref{chain-rule-D} and \eqref{identity-transformed}, 
 we obtain for every $t\in(0,\tilde{t})$ that 
\begin{align*}
		\frac{\frac{d^+}{dt}\tilde F(t)}{\tilde F(t)}\le&\frac{q'(t)}{q(t)^2}\bigg[\frac{N}{p}\log\bigg(\mathcal{L}_{p,N}{\sf AVR}_ {\sf m}^{-\frac{p}{N}}\int_X\vert \nabla w_t\vert^p\,d {\sf m}\bigg)-\frac{N}{p}\bigg(\frac{q(t)^2}{Nq'(t)\tilde F(t)^{q(t)}}\int_Xe^{q(t){\bf Q}_tu }{\vert \nabla{\bf Q}_tu\vert^p}\,d {\sf m}\bigg)\bigg]\\ \le&\frac{q'(t)}{q(t)^2}\cdot \frac{N}{p}\bigg[\log\bigg(\mathcal{L}_{p,N}{\sf AVR}_ {\sf m}^{-\frac{p}{N}}\frac{q(t)^p}{p^p\cdot \tilde F(t)^{q(t)}}\int_Xe^{q(t){\bf Q}_tu }\,\vert \nabla {\bf Q}_tu\vert^p\,d {\sf m}\bigg)\\&\quad\qquad\qquad-\log\bigg(\frac{e q(t)^2}{Nq'(t)\tilde F(t)^{q(t)}}\int_Xe^{q(t){\bf Q}_tu }{\vert \nabla{\bf Q}_tu\vert^p}\,d {\sf m}\bigg)\bigg]\\=&\frac{q'(t)}{q(t)^2}\cdot \frac{N}{p}\log\bigg(\frac{N\mathcal{L}_{p,N}{\sf AVR}_ {\sf m}^{-\frac{p}{N}}}{e p^p}\cdot q'(t)q(t)^{p-2}\bigg)\,.
\end{align*}
By integrating both sides on $(0,\tilde t)$, and taking into account that $t\mapsto \tilde F(t)$ is absolutely continuous, we obtain the inequality
\begin{align}\label{hyperc-estimate with F}
\tilde	F(\tilde{t})\le \tilde F(0)\cdot\bigg(\frac{\beta-\alpha}{\tilde{t}}\bigg)^{\frac{N}{p}\cdot\frac{\beta-\alpha}{\alpha\beta}}\cdot C_{\alpha,\beta,N,p, {\sf m}},
\end{align}
where
\begin{equation}\label{C-best-constant-hypercontractivity}
	C_{\alpha,\beta,p,N, {\sf m}}=\frac{\alpha^{\frac{N}{\alpha\beta}(\frac{\alpha}{p}+\frac{\beta}{p'})}}{\beta^{\frac{N}{\alpha\beta}(\frac{\beta}{p}+\frac{\alpha}{p'})}}\left({\sf AVR}_ {\sf m}\sigma_N (p')^\frac{N}{p'}\Gamma\left(\frac{N}{p'}+1\right)\right)^{\frac{\alpha-\beta}{\alpha\beta}},
\end{equation}
which concludes the proof of \eqref{hyperc-estimate}. \hfill$\square$

\subsection{Sharpness in the hypercontractivity estimate \eqref{hyperc-estimate}} 

  We assume by contradiction that one can improve the constant $\mathcal
L_{p,N}{\sf AVR}_ {\sf m}^{-\frac{p}{N}}\frac{Ne^{p-1}}{p^p}$ in \eqref{hyperc-estimate}, i.e., for convenience of computations,  there exists $\mathcal C<\mathcal
L_{p,N}{\sf AVR}_ {\sf m}^{-\frac{p}{N}}$ such that for every  $\beta>\alpha>0$,  $ t\in (0,t_0)$ and  $u\in \mathcal F_{t_0,x_0}(X)$ with $(1+d^{p'}(x_0,\cdot))e^{\alpha u}\in  L^1(X,{\sf m})$ and $e^\frac{\alpha u}{p}\in   W^{ 1,p}(X,d,{\sf m})$,  one has
\begin{equation}\label{hypercontractivity-contradiction}
	\|e^{{\bf Q}_{t}u}\|_{L^{\beta}(X,{\sf m})}\leq \|e^{u}\|_{L^{\alpha}(X,{\sf m})}\left(\frac{\beta-\alpha}{ t}\right)^{\frac{N}{p}\frac{\beta-\alpha}{\alpha\beta}}\frac{\alpha^{\frac{N}{\alpha\beta}(\frac{\alpha}{p}+\frac{\beta}{p'})}}{\beta^{\frac{N}{\alpha\beta}(\frac{\beta}{p}+\frac{\alpha}{p'})}}\left(\mathcal C\frac{Ne^{p-1}}{p^p}\right)^{\frac{N}{p}\frac{\beta-\alpha}{\alpha\beta}}.
\end{equation}

Fix  a function $u:X\to \mathbb R$ with the above properties, a number $\alpha>0$  and the function $\beta:=\beta(t)=\alpha+y t$ with $t>0$, where $y>0$ will be specified later. With these choices, inequality  \eqref{hypercontractivity-contradiction} can be equivalently written into the form 
\begin{equation}\label{limit-before}
		\frac{\log\Vert e^{{\bf Q}_tu}\Vert_{L^{\beta(t)}(X,{\sf m})}-\log\Vert e^{u}\Vert_{L^{\alpha}(X,{\sf m})}}{t}\le \frac{1}{t}\log\bigg[\bigg({y\mathcal C}\frac{Ne^{p-1}}{p^p}\bigg)^{\frac{Nyt}{p\alpha\beta( t)}}\cdot \frac{\alpha^{\frac{N}{\alpha\beta(t)}(\frac{\alpha}{p}+\frac{\beta(t)}{p'})}}{\beta(t)^{\frac{N}{\alpha\beta(t)}(\frac{\beta(t)}{p}+\frac{\alpha}{p'})}}\bigg],
\end{equation}
for every $t\in (0,t_0)$. If we introduce 
$$
\overline F(t):=\bigg(\int_Xe^{\beta(t){\bf Q}_tu }d {\sf m}\bigg)^{1/\beta(t)}=\Vert e^{{\bf Q}_tu}\Vert_{L^{\beta(t)}(X,{\sf m})},\ t\in [0,t_0),
$$
and by taking the limit $t\to 0^+$ in \eqref{limit-before}, we obtain that 
\begin{equation}\label{limit-before-2}
	\frac{\frac{d^+}{dt}\overline F(t)}{\overline F(t)}\bigg|_{t=0}\le \lim_{t\to 0^+}\log\left[\bigg({y\mathcal C}\frac{Ne^{p-1}}{p^p}\bigg)^{\frac{Ny}{p\alpha\beta( t)}}\cdot \left(\frac{\alpha^{\frac{N}{\alpha\beta(t)}(\frac{\alpha}{p}+\frac{\beta(t)}{p'})}}{\beta(t)^{\frac{N}{\alpha\beta(t)}(\frac{\beta(t)}{p}+\frac{\alpha}{p'})}}\right)^\frac{1}{t}\right].
\end{equation}
A similar computation  as in \eqref{right-derivative} and  Proposition \ref{proposition-hyper}/(iv), show  that  
\begin{eqnarray}
\nonumber	\frac{\frac{d^+}{dt}\overline F(t)}{\overline F(t)}\bigg|_{t=0}&=&\frac{1}{\beta(t)^2}\bigg(-\beta'(t)\log\big(\overline F(t)^{\beta(t)}\big)+\frac{\beta(t)^2}{\overline F(t)^{\beta(t)}}\int_Xe^{\beta(t){\bf Q}_tu }\frac{d^+}{dt}{\bf Q}_tu\,d {\sf m}\\&&\nonumber+\frac{\beta'(t)}{\overline F(t)^{\beta(t)}}\mathcal{E}\big(e^{\beta(t){\bf Q}_tu }\big)\bigg)\bigg|_{t=0}\\&\geq &\nonumber \frac{y}{\alpha^2\|e^u\|^\alpha_{L^{\alpha}(X,{\sf m})}}\left(\mathcal{E}\big(e^{\alpha u}\big)-\|e^u\|^\alpha_{L^{\alpha}(X,{\sf m})}\log (\|e^u\|^\alpha_{L^{\alpha}(X,{\sf m})})-\frac{\alpha^2}{py}\int_Xe^{\alpha u}|Du|^pd {\sf m}\right).
\end{eqnarray}
The latter estimate and the limit 
$$\lim_{t\to 0^+}\left(\frac{\alpha^{\frac{N}{\alpha\beta(t)}(\frac{\alpha}{p}+\frac{\beta(t)}{p'})}}{\beta(t)^{\frac{N}{\alpha\beta(t)}(\frac{\beta(t)}{p}+\frac{\alpha}{p'})}}\right)^\frac{1}{t}=e^\frac{Ny((p-2)\log \alpha -p)}{\alpha^2 p}$$ transform inequality \eqref{limit-before-2} into 
$$\frac{\mathcal{E}\big(e^{\alpha u}\big)}{\|e^u\|^\alpha_{L^{\alpha}(X,{\sf m})}}-\log (\|e^u\|^\alpha_{L^{\alpha}(X,{\sf m})})\leq \frac{\alpha^2}{py}\frac{\displaystyle \int_Xe^{\alpha u}|\nabla u|^pd {\sf m}}{\|e^u\|^\alpha_{L^{\alpha}(X,{\sf m})}}+\frac{N}{p}\left(\log\left({y\mathcal C}\frac{Ne^{p-1}}{p^p}\right)+(p-2)\log \alpha -p\right);$$
here we also used relation \eqref{Cheeger-equality}, i.e., $|Du|=|\nabla u|$ ${\sf m}$-a.e., based on  the fact that $u\in \mathcal F_{t_0,x_0}(X)\subset {\rm Lip}_{\rm loc}(X)$. 
By assumption, $e^\frac{\alpha u}{p}\in   W^{ 1,p}(X,d,{\sf m})$, thus  
$$w:=\frac{e^\frac{\alpha u}{p}}{\|e^u\|^\frac{\alpha}{p}_{L^{\alpha}(X,{\sf m})}}\in  W^{ 1,p}(X,d,{\sf m})\ \ {\rm and}\ \ \displaystyle\int_X w^{p}d{\sf m}=1. $$
 With this notation, the latter inequality can be equivalently written into the form
\begin{equation}\label{almost-log-SOb}
	\int_Xw^p \log w^p\leq \frac{\alpha^2}{py}\left(\frac{p}{\alpha}\right)^p \int_X|\nabla w|^pd {\sf m}+\frac{N}{p}\left(\log\left({y\mathcal C}\frac{Ne^{p-1}}{p^p}\right)+(p-2)\log \alpha -p\right).
\end{equation}
Minimizing the right hand side of \eqref{almost-log-SOb} in $y>0$, its optimal value  is
$y=\frac{\alpha^2}{N}\left(\frac{p}{\alpha}\right)^p \displaystyle\int_X|\nabla w|^pd {\sf m}.$
Replacing this value into \eqref{almost-log-SOb}, it follows that
$$\int_Xw^p \log w^p\leq  \frac{N}{p}\log\left({\mathcal C}\int_X|\nabla w|^pd {\sf m}\right).$$
Due to the sharpness in Theorem \ref{log-Sobolev-main}, see \eqref{C-estimate}, we have that 
$
\mathcal C \geq  \mathcal L_{p, N}{\sf AVR}_ {\sf m}^{-\frac{p}{N}}$, which contradicts our initial assumption $\mathcal C<\mathcal
L_{p,N}{\sf AVR}_ {\sf m}^{-\frac{p}{N}}$. \hfill $\square$

 \section{Gaussian log-Sobolev inequality in ${\sf RCD}(0,N)$ spaces: proof of Theorem \ref{Gaussian-theorem}}\label{Gaussian-section}
 
 In this section we are going to prove Theorem \ref{Gaussian-theorem}; to do this, we need some basic properties of ${\sf RCD}(0,N)$ spaces that are presented in the subsequent subsection. 
 
 \subsection{Properties of ${\sf RCD}(0,N)$ spaces}

 A  metric measure space $(X,d, {\sf m})$  satisfies the Riemannian curvature-dimension condition ${\sf RCD}(0,N)$ for $N>1$, if it is a ${\sf CD}(0,N)$ space and it is infinitesimally Hilbertian, i.e., the Banach space $W^{1,2}(X, d,{\sf m})$ is Hilbertian. Due to Cavalletti and Milman \cite{Cavalletti-MIlman}, this definition of ${\sf RCD}(0,N)$ is equivalent to the previously introduced notions of ${\sf RCD}^{e}(0,N)$ by  Erbar, Kuwada and Sturm \cite{EKS} and of  ${\sf RCD}^{\star}(0,N)$ by Ambrosio, Mondino and Savar\'e \cite{AMondinoSavare}. The first form of the Riemannian curvature-dimension condition was introduced by Ambrosio, Gigli and Savar\'e \cite{AGS-Duke}, with no upper bound on the dimension. Typical examples of ${\sf RCD}(0,N)$ spaces  include measured Gromov--Hausdorff limit spaces of Riemannian manifolds with non-negative Ricci curvature. 
 
 In the sequel, we fix  a  metric measure space $(X,d, {\sf m})$  satisfying  ${\sf RCD}(0,N)$ for some $N>1$. 
 The infinitesimal Hilbertianity of $(X,d, {\sf m})$ can be characterized  by its 2-infinitesimal strict convexity and the symmetry of the map $$(f,g)\mapsto Df(\nabla g)=D^\pm f(\nabla g),$$
  i.e., 
 \begin{equation}\label{symmetry}
 	Df(\nabla g)=Dg(\nabla f),
 \end{equation} 
 where 
 $$D^+f(\nabla g)=\inf_{\varepsilon>0}\frac{|\nabla(g+\varepsilon f)|^2-|\nabla g|^2}{2\varepsilon}\ \ {\rm and}\ \ D^-f(\nabla g)=\sup_{\varepsilon<0}\frac{|\nabla(g+\varepsilon f)|^2-|\nabla g|^2}{2\varepsilon}$$
 for every $f,g:X\to \mathbb R$ with finite 2-Cheeger energy,
see Gigli \cite[\S 4.3]{Gigli1}. 
 Note that in particular, for every $f:X\to \mathbb R$ with finite 2-Cheeger energy, one has 
  \begin{equation}\label{inner-norm}
 	Df(\nabla f)=|\nabla f|^2.
 \end{equation}
According to \eqref{symmetry}, we shall use the formal  notation $\nabla f\cdot \nabla g$ for the object $Df(\nabla g)$; as a consequence, the linearity  of the differential operator $D$, combined with the symmetry property \eqref{symmetry}, implies the bilinearity of $(f,g)\mapsto \nabla f\cdot \nabla g$. In addition, we also have the Leibnitz-type rule
\begin{equation}\label{Leibnitz-rule}
	\nabla (f_1f_2)\cdot \nabla g= f_1	\nabla f_2\cdot \nabla g + f_2	\nabla f_1\cdot \nabla g\ \ \  {\sf m}{\rm -a.e.}, 
\end{equation}
for every $f_1,f_2,g\in L_{\rm loc}^\infty(X)$ having locally finite 2-Cheeger energies. If $f_1,f_2:X\to \mathbb R$ are two such functions, as a simple consequence of relations \eqref{symmetry}-\eqref{Leibnitz-rule} we have that
\begin{equation}\label{multiplication-norm}
|\nabla (f_1 f_2)|^2=f_1^2|\nabla f_2|^2 + 2f_1f_2\nabla f_1\cdot \nabla f_2+ f_2^2|\nabla f_1|^2.
\end{equation}
 
 
 It turns out that in the framework of ${\sf RCD}(0, N)$ spaces, a second order calculus related to the Laplacian can be developed. Let us recall that the \textit{Laplacian} $\Delta\colon D(\Delta)\to L^2(X, {\sf m})$ is a densely defined linear operator whose domain $D(\Delta)$ consists of all functions $f\in W^{1,2}(X,d, {\sf m})$ satisfying
 \begin{align*}
 	\int_Xh  g\ d {\sf m}=-\int_X\nabla h\cdot\nabla f\ d {\sf m}\ \text{ for any }h\in W^{1,2}(X,d, {\sf m})\,,
 \end{align*}
 and the unique $g\in L^{2}(X, {\sf m})$ that satisfies this property is denoted by $\Delta f$. 
 More generally, we say that $f\in W^{1,2}_{\mathrm{loc}}(X,d, {\sf m})$ is in the domain of the \textit{measure-valued Laplacian}, and we write $f\in D(\mathbf{\Delta})$, if there exists a Radon measure $\nu$ on $X$ such that
 \begin{align*}
 	\int_X\psi\ d\nu=-\int_X\nabla f\cdot\nabla\psi\ d {\sf m}\ \text{ for any }\psi\in\mathrm{Lip}_c(X)\,.
 \end{align*}
 In this case, we write $\mathbf{\Delta}f:=\nu$. If moreover $\mathbf{\Delta}f\ll {\sf m}$ with $L^2_{\mathrm{loc}}$-density, we denote by $\Delta f$ the unique function in $L^2_{\mathrm{loc}}(M,d, {\sf m})$ such that $\mathbf{\Delta}f=\Delta f {\sf m}$.
 
 Using this notation, one has the following \textit{sharp Laplacian comparison} estimate for the distance function on an arbitrary ${\sf RCD}(0,N)$ space  $(X,d, {\sf m})$, see e.g. Gigli \cite[Corollary 5.15]{Gigli1}:   if we consider the distance function $d_{x_0}\colon X\to [0,\infty)$ with $d_{x_0}(x):=d(x_0,x)$ for some $x_0\in X$, we have that
\begin{equation}\label{Lap-comparison}
	\frac{d_{x_0}^2}{2}\in D(\mathbf{\Delta})\ \text{ and }\ \mathbf{\Delta}\frac{d_{x_0}^2}{2}\le N {\sf m} .
\end{equation}

After this preparation we are ready to present the proof of Theorem \ref{Gaussian-theorem}.

 \subsection{Proof of the Gaussian log-Sobolev inequality \eqref{Gaussian-log-Sobolev}}
 
 Let  $x_0\in X$ be fixed. We recall from the Introduction the ${\sf m}$-{Gaussian measure} 
\begin{equation}\label{mGx0-measure}
	d{\sf m}_{G,x_0}(x):=G^{-1}e^{-\frac{d^2(x_0,x)}{2}}d{\sf m}(x),
\end{equation}
where
$$ \displaystyle G=\int_Xe^{-\frac{d^2(x_0,x)}{2}}d {\sf m}(x)>0.$$
We first notice that $G<+\infty$; indeed, by the layer cake representation 	one has
\begin{align}\label{C const}
	G=\int_X e^{-\frac{d^2(x_0,x)}{2}}d {\sf m}(x)=\int_0^\infty  {\sf m}\big(B(x_0,\rho)\big)\rho e^{-\frac{\rho^2}{2}}d\rho,
\end{align}
thus, by the Bishop--Gromov inequality,
if $\theta_ {\sf m}(x_0)=+\infty$, then 
$$G\leq \int_0^1  {\sf m}\big(B(x_0,\rho)\big)\rho e^{-\frac{\rho^2}{2}}d\rho+{\sf m}\big(B(x_0,1))\int_1^\infty  \rho^{N+1} e^{-\frac{\rho^2}{2}}d\rho<+\infty,$$
while  if $\theta_ {\sf m}(x_0)<+\infty$, then  
\begin{equation}\label{estimate-f-finite}
	G\leq  \theta_ {\sf m}(x_0)\sigma_N \int_0^\infty  \rho^{N+1} e^{-\frac{\rho^2}{2}}d\rho=\theta_ {\sf m}(x_0)\sigma_N2^\frac{N}{2}\Gamma(\frac{N}{2}+1)
<+\infty. 
\end{equation}
 If $\theta_{\sf m}(x_0)=+\infty$, inequality \eqref{Gaussian-log-Sobolev} trivially holds; therefore, we may assume that $\theta_{\sf m}(x_0)<+\infty$.
%
%

Fix $u\in W^{1,2}(X,d,{\sf m}_{G,x_0})$ such that $$\displaystyle \int_Xu^2 d {\sf m}_{G,x_0}=1,$$ and let us consider the function $$v(x):=u(x)\cdot\sqrt{ \rho_{G,x_0}(x)},\ x\in X,$$ 
where $$\rho_{G,x_0}(x)=G^{-1}e^{-\frac{d^2_{x_0}(x)}{2}}$$ is the density of $d{\sf m}_{G,x_0}$ with respect to $d{\sf m}$. By definition of $v$, we can write 
\begin{equation}\label{v-normalized}
	\int_Xv^2d {\sf m}=\int_Xu^2\rho_{G,x_0} d {\sf m}=\int_Xu^2 d{\sf m}_{G,x_0}=1.
\end{equation}
We claim that $v\in  W^{1,2}(X,d,{\sf m})$. Indeed, by relation \eqref{multiplication-norm}, the eikonal equation \eqref{eikonal}, and the non-smooth chain rule, we have ${\sf m}$-a.e. that 
\begin{eqnarray}\label{estimate-gauss-1}
\nonumber	\vert\nabla v\vert^2&=&\vert\nabla (u\sqrt{\rho_{G,x_0}})\vert^2=\frac{1}{4}u^2d_{x_0}^2\rho_{G,x_0}\vert\nabla d_{x_0}\vert^2-\rho_{G,x_0} u \nabla u\cdot \nabla \left(\frac{d_{x_0}^2}{2}\right)+\rho_{G,x_0}\vert\nabla u\vert^2\\&= &\nonumber\frac{1}{4}u^2d_{x_0}^2\rho_{G,x_0}-  \nabla \left(\frac{u^2}{2}\rho_{G,x_0}\right)\cdot \nabla \left(\frac{d_{x_0}^2}{2}\right)+ \frac{u^2}{2}\nabla \rho_{G,x_0}\cdot \nabla \left(\frac{d_{x_0}^2}{2}\right)+\vert\nabla u\vert^2\rho_{G,x_0}\\&= &-\frac{1}{4}u^2d_{x_0}^2\rho_{G,x_0}-  \nabla \left(\frac{u^2}{2}\rho_{G,x_0}\right)\cdot \nabla \left(\frac{d_{x_0}^2}{2}\right)+\vert\nabla u\vert^2\rho_{G,x_0}. 
\end{eqnarray}
Due to the Laplacian comparison property \eqref{Lap-comparison}, one has that 
\begin{equation}\label{Lap-comp-applied}
	-\int_X\nabla \left(\frac{u^2}{2}\rho_{G,x_0}\right)\cdot \nabla \left(\frac{d_{x_0}^2}{2}\right)d{\sf m}=\int_X \frac{u^2}{2}\rho_{G,x_0} \Delta \left(\frac{d_{x_0}^2}{2}\right)d{\sf m}\leq \frac{N}{2}\int_X {u^2}\rho_{G,x_0}d{\sf m}=\frac{N}{2}.
\end{equation}
Integrating  \eqref{estimate-gauss-1} with respect to $d{\sf m}$ and given this final estimate, together with the fact that 
$u\in W^{1,2}(X,d,{\sf m}_{G,x_0})$ (and hence $\displaystyle \int_X\vert\nabla u\vert^2d {\sf m}_{G,x_0}<\infty$),  we obtain that 
$$\displaystyle \int_X\vert\nabla v\vert^2d {\sf m}<\infty.$$
Using this fact and \eqref{v-normalized}, we obtain the claim, i.e.,  $v\in  W^{1,2}(X,d,{\sf m})$. 

Furthermore,  integrating again \eqref{estimate-gauss-1} and using the above observations, we also see that $$\int_Xu^2d_{x_0}^2\rho_{G,x_0}d{\sf m}=\int_Xu^2d_{x_0}^2d{\sf m}_{G,x_0}<\infty.$$
Summing up the above estimates, we obtain that
\begin{equation}\label{ell-ell}
	\int_X\vert\nabla v\vert^2d {\sf m}\le \int_X\vert \nabla u\vert^2d {\sf m}_{G,x_0} +\frac{N}{2}-\frac{1}{4}\int_Xu^2d^2 d {\sf m}_{G,x_0}\,,
\end{equation}
and since $v\in  W^{1,2}(X,d,{\sf m})$ verifies \eqref{v-normalized}, by the 
$L^2$-log-Sobolev inequality \eqref{LSI} one has 
\begin{align}\label{ineq on v}
	\int_X v^2\log v^2 d {\sf m}\le \frac{N}{2}\log\bigg(\frac{2}{ eN}\left(\sigma_N\Gamma(\frac{N}{2}+1){\sf AVR}_ {\sf m}\right)^{-\frac{2}{N}}\int_X\vert\nabla v\vert^2d {\sf m}\bigg). 
\end{align}
Writing  the left-hand side of \eqref{ineq on v} in terms of $u$, we obtain
	\begin{eqnarray}\label{expanding}
	\nonumber	\int_Xv^2\log v^2d {\sf m}&=&\int_X u^2\log(u^2 \rho_{G,x_0}) \rho_{G,x_0}d {\sf m}\\&=&\int_Xu^2\log u^2 d {\sf m}_{G,x_0}-\log G-\frac{1}{2}\int_Xu^2{d_{x_0}^2}d {\sf m}_{G,x_0}.
	\end{eqnarray}
	If we apply the elementary inequality $\log (ey)\leq y$, which holds for every $y>0$,  the right-hand side of \eqref{ineq on v} can be estimated as 
	\begin{align*}
		\frac{N}{2}\log\bigg(&\frac{1}{2 e^2}\left(\sigma_N\Gamma\big(\frac{N}{2}+1\right){\sf AVR}_ {\sf m}\big)^{-\frac{2}{N}}\bigg)+\frac{N}{2}\log\bigg(\frac{4e}{N}\int_X\vert\nabla v\vert^2d {\sf m}\bigg)\le \\&\frac{N}{2}\log\bigg(\frac{1}{2 e^2}\left(\sigma_N\Gamma\big(\frac{N}{2}+1\right){\sf AVR}_ {\sf m}\big)^{-\frac{2}{N}}\bigg)+2\int_X\vert\nabla v\vert^2d {\sf m}.
	\end{align*}
The latter estimate together with relations  \eqref{ell-ell}-\eqref{expanding} and \eqref{estimate-f-finite} imply  \eqref{Gaussian-log-Sobolev}, i.e., 
	$$
	\int_Xu^2\log u^2 d {\sf m}_{G,x_0}\le 2\int_X\vert\nabla u\vert^2 d {\sf m}_{G,x_0}+\log\frac{\theta_{\sf m}(x_0)}{{\sf AVR}_ {\sf m}}.
$$

\begin{remark}\rm 
	 As we pointed out in the introduction, the  Gaussian log-Sobolev inequality \eqref{curved-Gross-inequality} holds with the measure $d \gamma_V = e^{-V} d{\rm vol}$ whenever $V$ verifies a strong convexity assumption, see   Cordero-Erausquin, McCann and Schmuckenschl\"{a}ger \cite{CEMS2}. Note that the Gaussian log-Sobolev inequality \eqref{Gaussian-log-Sobolev} cannot be obtained via the Pr\'ekopa--Leindler inequality, as the function $V(x)=\frac{1}{2}d^2(x_0,x)$ does not provide any reasonable convexity property (even in the setting of Riemannian manifolds with nonnegative Ricci curvature).\ Such convexity property occurs more naturally on Cartan-Hadamard manifolds, the threshold geometric objects being the Euclidean spaces.  However, in the setting of ${\sf RCD}(0,N)$ spaces the additive term log$\frac{\theta_{\sf m}(x_0)}{{\sf AVR}_ {\sf m}}$ balances the lack of such convexity.  
\end{remark}

\subsection{Sharpness of the constant 2 in the Gaussian log-Sobolev inequality \eqref{Gaussian-log-Sobolev}}

We are going to prove that the constant 2 in \eqref{Gaussian-log-Sobolev}  is sharp  (whenever we assume  that $\theta_{\sf m}(x_0)<+\infty$).
By contradiction, let us assume that there exists $C\in (0,2)$ such that, for every $\displaystyle \int_Xu^2 d {\sf m}_{G,x_0}=1$, one has 
\begin{equation}\label{C-contradiction}
	\int_Xu^2\log u^2 d {\sf m}_{G,x_0}\le C\int_X\vert\nabla u\vert^2 d {\sf m}_{G,x_0}+\log\frac{\theta_{\sf m}(x_0)}{{\sf AVR}_ {\sf m}}.
\end{equation}
For every $\lambda>0$, let	$$
f(\lambda):=\displaystyle\int_X e^{-{\lambda}
	{{d}}(x_0,x)^{2}}d {\sf m}(x).
$$ 
Recalling that 
$$ \displaystyle G=\int_Xe^{-\frac{d^2(x_0,x)}{2}}d {\sf m}(x)>0,$$
consider the function  $$u_\lambda(x)=
\left(\frac{G}{f(\lambda)}\right)^\frac{1}{2}e^{(\frac{1}{4}-\frac{\lambda}{2})d(x_0,x)^2},\ \ x\in X.$$
One can easily observe that
$\displaystyle \int_Xu_\lambda^2 d {\sf m}_{G,x_0}=1$. Thus, we can use $u_\lambda$ in \eqref{C-contradiction}.  By applying the eikonal equation \eqref{eikonal}, a straightforward reorganization of the terms provide the inequality 
\begin{equation}\label{C-inequality}
	-\log f(\lambda)\leq \frac{1}{f(\lambda)} \left(\frac{1}{2}-\lambda\right)\left(\frac{C}{2}-1-C\lambda\right)\displaystyle\int_X e^{-{\lambda}
		{{d}}(x_0,x)^{2}}{{d}}(x_0,x)^{2}d {\sf m}(x)+\log\frac{\theta_{\sf m}(x_0)}{G\cdot{\sf AVR}_ {\sf m}},\ \ \forall \lambda>0.
\end{equation}  
By the layer cake representation one has that 
\begin{eqnarray}\label{first-estimate-0}
	\nonumber \lim_{\lambda\to 0^+}\lambda^\frac{N}{2}f(\lambda)&=&2\lim_{\lambda\to 0^+}\lambda^{\frac{N}{2}+1}\int_0^\infty {\sf m}(B(x_0,\rho))\rho e^{-\lambda \rho^2}d\rho\\&=&\nonumber \lim_{\lambda\to 0^+}\lambda^{\frac{N}{2}}\int_0^\infty {\sf m}(B(x_0,\sqrt{t/\lambda})) e^{-t}dt\ \ \ \quad [{\rm change\ of\ var.} \ \rho=\sqrt{t/\lambda}]\\&=& 
	\nonumber \sigma_N\lim_{\lambda\to 0^+}\int_0^\infty \frac{{\sf m}(B(x_0,\sqrt{t/\lambda}))}{\sigma_N(\sqrt{t/\lambda})^N} t^{\frac{N}{2}}e^{-t}dt\\&=& \sigma_N{\sf AVR}_ {\sf m}\Gamma\left(\frac{N}{2}+1\right)=
	\pi^\frac{N}{2}{\sf AVR}_ {\sf m},
\end{eqnarray}
where we also used the monotone convergence theorem (together with the Bishop--Gromov comparison principle) and the definition of ${\sf AVR}_ {\sf m}.$

In a similar way, one has that
$$\displaystyle\int_X e^{-{\lambda}
	{{d}}(x_0,x)^{2}}{{d}}(x_0,x)^{2}d {\sf m}(x)=2\int_0^\infty {\sf m}(B(x_0,\rho))\rho (\lambda\rho^2-1) e^{-\lambda \rho^2}d\rho,$$
and 
\begin{equation}\label{second-estimate}
	\lim_{\lambda\to 0^+}\lambda^{\frac{N}{2}+1}\displaystyle\int_X e^{-{\lambda}
		{{d}}(x_0,x)^{2}}{{d}}(x_0,x)^{2}d {\sf m}(x)=\frac{N}{2}\pi^\frac{N}{2}{\sf AVR}_ {\sf m}.
\end{equation}
Multiplying \eqref{C-inequality} by $\lambda>0$, and letting $\lambda\to 0^+$, on account of relations \eqref{first-estimate-0} and \eqref{second-estimate} we obtain that  
$0\leq \frac{N}{4}\left(\frac{C}{2}-1\right),$
i.e., $C\geq 2$, which contradicts our initial assumption $C\in (0,2)$. \hfill $\square$\\

Inspired by Bobkov,  Gentil and Ledoux \cite[Remark 2.2]{BobkovGL},  we establish a hypercontractivity  bound for the ${\sf m}$-{Gaussian probability measure} ${\sf m}_{G,x_0}$, see \eqref{mGx0-measure}. In the sequel, the previous notions are considered for $p=2$. 

	\begin{corollary}\label{Gaussian-theorem-rcd}
	Let $(X,d, {\sf m})$ be an ${\sf RCD}(0,N)$ space with $N>1$ and assume that ${\sf AVR}_ {\sf m}>0$. Given $\alpha>0$, $x_0\in X$,  for any function $u:X\to \mathbb R$ with $e^\frac{u}{2}\in W^{1,2}(X,d,{\sf m}_{G,x_0})$ and $t\geq 0$ one has  that
	\begin{equation}\label{Gaussian-hyper-rcd}
			\|e^{{\bf Q}_{t}u}\|_{L^{\alpha+t}(X,{\sf m}_{G,x_0})}\leq e^{H_{\alpha,t,x_0}} \|e^{u}\|_{L^{\alpha}(X,{\sf m}_{G,x_0})},
	\end{equation}
where $H_{\alpha,t,x_0}=\frac{t}{\alpha(\alpha+t)}\log\frac{\theta_{\sf m}(x_0)}{{\sf AVR}_ {\sf m}}.$
\end{corollary}
\noindent {\it Proof.} 
For any $u\in W^{1,2}(X,d,{\sf m}_{G,x_0})$ inequality \eqref{Gaussian-log-Sobolev} can be written into the equivalent form 
\begin{equation}\label{extended-entropy}
	\int_Xu^2\log u^2 d {\sf m}_{G,x_0}-\displaystyle \int_Xu^2 d {\sf m}_{G,x_0}\log\displaystyle \int_Xu^2 d {\sf m}_{G,x_0}\le 2\int_X\vert\nabla u\vert^2 d {\sf m}_{G,x_0}+\log\frac{\theta_{\sf m}(x_0)}{{\sf AVR}_ {\sf m}}\displaystyle \int_Xu^2 d {\sf m}_{G,x_0}.
\end{equation}
If we fix $t\geq 0$  and any function $u:X\to \mathbb R$ with $e^\frac{u}{2}\in W^{1,2}(X,d,{\sf m}_{G,x_0})$,   a similar argument as in subsection \ref{subsection-hypercontractivity} combined with \eqref{extended-entropy} yield the required estimate  \eqref{Gaussian-hyper-rcd}. 
\hfill $\square$
 
 \begin{remark}\rm  As expected, \eqref{Gaussian-hyper-rcd} implies  (as $t\to 0$)   the validity of the  Gaussian log-Sobolev inequality \eqref{extended-entropy}. 
 \end{remark}

\section{Final remarks}\label{section-final}

At the end, we indicate further perspectives related to our results, which can be considered as starting points for forthcoming investigations.

\begin{itemize}
	\item[I.] \textit{Rigidities.} Since the log-Sobolev inequality in Theorem \ref{log-Sobolev-main} is sharp, a natural question arises: can one characterize the \textit{equality} case in \eqref{LSI}? The proof of Theorem \ref{log-Sobolev-main} shows that if equality holds in \eqref{LSI}, we necessarily have equality in the P\'olya--Szeg\H o inequality \eqref{Polya-Szego}, therefore the rigidity of the sharp isoperimetric inequality \eqref{eqn-isoperimetric}
	 is \textit{expected} to hold giving in turn that $X$ is an Euclidean metric measure cone. However, according to Nobili and Violo \cite{NV}, some regularity of the extremal function is a priori needed in order to
	 characterize the equality in the P\'olya--Szeg\H o inequality \eqref{Polya-Szego}. Therefore, although this approach seems to be promising, technical difficulties prevent to
	 study the rigidity scenario in the sharp log-Sobolev inequality \eqref{LSI}.\ We notice that several rigidity results are available in the literature concerning the equality in the sharp isoperimetric inequality \eqref{eqn-isoperimetric}; beside the smooth setting, see Agostiniani,  Fogagnolo and  Mazzieri \cite{Agostiniani-etal, Fogagnolo-Maz-JFA}, Balogh and Krist\'aly \cite{BK}, Brendle \cite{Brendle} and  Johne \cite{Johne}, there are recent achievements also in the nonsmooth setting by 
	Antonelli,  Pasqualetto, Pozzetta and Semola \cite{Antonellietal} for  non-collapsed {\sf RCD} spaces, as well as Cavalletti and Manini \cite{CM-2} for possibly collapsing spaces. 
	
	\item[II.] ${\sf MCP}(0,N)$ \textit{spaces}. In a recent paper, Cavalletti and Manini \cite{CM} proved a sharp isoperimetric inequality on  ${\sf MCP}(0,N)$ spaces; namely, if  $(X,d, {\sf m})$ is a essentially non-branching metric measure space satisfying the ${\sf MCP}(0,N)$ condition for some $N>1$ (see Ohta \cite{Ohta} or Sturm \cite{Sturm-2} for the definition) and  ${\sf AVR}_ {\sf m}>0 $, one has that, for every bounded Borel subset $\Omega\subset X$,  
	\begin{equation}\label{eqn-isoperimetric-MCP}
		{\sf m}^+(\Omega)\geq \left(N\sigma_N{\sf AVR}_ {\sf m}\right)^\frac{1}{N} {\sf m}(\Omega)^\frac{N-1}{N},
	\end{equation}
	and the constant $\left(N\sigma_N{\sf AVR}_ {\sf m}\right)^\frac{1}{N}$ in  \eqref{eqn-isoperimetric-MCP} is sharp. It is worth noticing that this constant in \eqref{eqn-isoperimetric-MCP} is  slightly worse than the one from  \eqref{eqn-isoperimetric}. The isoperimetric inequality \eqref{eqn-isoperimetric-MCP} together with a suitable co-area formula on ${\sf MCP}(0,N)$ should yield a P\'olya--Szeg\H o inequality. Furthermore, a sharp  log-Sobolev inequality on the 1-dimensional ${\sf MCP}$ model case is needed (see \cite[\S 3.2]{CM}), similar to \eqref{log-Sob-1D} as in the ${\sf CD}$ setting. 
	
	\item[III.] \textit{Gaussian log-Sobolev inequality on ${\sf CD}(0,N)$ spaces.} One of the crucial steps in the proof of Theorem \ref{Gaussian-theorem} is the estimate \eqref{estimate-gauss-1}, which deeply explores the  infinitesimal Hilbertianity of the  ${\sf RCD}(0,N)$ spaces. It is not clear at this moment if the proof can be extended to ${\sf CD}(0,N)$ (for instance, to non-Riemannian Finsler structures), in spite of the fact that further indispensable ingredients are already know to be valid on such structures, as versions of the Laplace comparison of the distance functions similar to \eqref{Lap-comparison}, see Cavalletti and Mondino \cite{CMondino} and Gigli \cite{Gigli1}.
	
\end{itemize}

\vspace{0.2cm}

\noindent {\bf Acknowledgment.} A. Krist\'aly thanks T. Ill\'es, M. Pint\'er, M. E.-Nagy, P. Rig\'o and Z. Sz\'ant\'o, for  their kind invitation and  hospitality during his visit as a senior research fellow at the Corvinus Centre for Operations Research, Corvinus Institute for Advanced Studies, Corvinus University of Budapest, Budapest, Hungary.  The authors thank the reviewers for their careful reading and helpful comments.

%

\end{document}